\documentclass[a4paper]{amsart}

\usepackage[utf8]{inputenc}
\usepackage{amssymb,mathrsfs,amscd}
\usepackage[T1]{fontenc}
\usepackage[totalwidth=15cm,totalheight=22cm]{geometry}
\usepackage{esint}
\usepackage{enumerate}
\usepackage{xcolor}
\newtheorem{cor}{Corollary}[section]
\newtheorem{thm}[cor]{Theorem}
\newtheorem{prop}[cor]{Proposition}
\newtheorem{lemma}[cor]{Lemma}

\theoremstyle{definition}
\newtheorem{defi}[cor]{Definition}
\theoremstyle{remark}
\newtheorem{remark}[cor]{Remark}

\renewcommand{\H}{{\mathscr{H}}}

\newcommand{\R}{{\mathbb R}}

\renewcommand{\S}{{\mathbb S}}

\newcommand{\C}{\mathcal{C}}

\newcommand{\hes}{\mbox{Hess}}
\def\res{\mathop{\hbox{\vrule height 7pt width .5pt depth 0pt \vrule height .5pt width 6pt depth 0pt}}\nolimits}
\newcommand{\be}{\begin{eqnarray}}
\newcommand{\ee}{\end{eqnarray}}

\renewcommand{\tilde}[1]{\widetilde{#1}}
\renewcommand{\bar}[1]{\overline{#1}}
\newcommand{\Leb}[1]{{\mathscr L}^{#1}}
\newcommand{\ep}{\varepsilon}

\renewcommand{\hat}[1]{\widehat{#1}}

\newcommand{\tphi}{\hat \phi}
\newcommand{\hd}{\hat d}

\renewcommand{\j}{\underline{j}}
\renewcommand{\i}{\underline{i}}
\newcommand{\Cw}{{C_{\it w}}}
\newcommand{\Cwo}{{C_{\it w,o}}}
\newcommand{\GBV}{{\mathcal{BV}}}
\newcommand{\GM}{{\mathcal{GM}}}
\newcommand{\Si}{\mathcal{S}}
\renewcommand{\angle}{\measuredangle}
\begin{document}

\title{DC Calculus}
\author{Luigi Ambrosio}
\address{Scuola Normale Superiore\\  
    Piazza dei Cavalieri 7 
         56126 Pisa, Italy    }
\email{luigi.ambrosio@sns.it}

\author{J\'er\^ome Bertrand}
\address{Institut de Math\'ematiques de Toulouse, UMR CNRS 5219 \\
Universit\'e Toulouse III \\
31062 Toulouse cedex 9, France}
\email{bertrand@math.univ-toulouse.fr}

%%%%%%%%%%   body text  %%%%%%%%%%%%%%%

\begin{abstract}
In this paper, we extend the DC Calculus introduced by Perelman on finite dimensional Alexandrov spaces with curvature bounded below. 
Among other things, our results allow us to define the Hessian and the Laplacian of DC functions (including distance functions as a particular instance) as a measure-valued tensor and a Radon measure respectively. We show that these objects share various properties with their analogues on smooth Riemannian manifolds.
\end{abstract}
\maketitle
\tableofcontents

\section{Introduction}
%In this introduction, we describe in an informal way the results we prove in this paper. All the necessary definitions and precise statements are then provided in the core of the paper. We believe this loose exposition makes it easier to compare with earlier results and highlight the main novelties we provide.    

In this paper, we investigate the $DC$ Calculus introduced by Perelman in \cite{DC} on the manifold part $X^*$ of a finite dimensional Alexandrov space $X$ with curvature bounded below. The term ``$DC$'' stands for difference of concave functions. The main point is that, contrary to the set of (semi-)concave functions, the set of $DC$ functions on Euclidean space is an algebra. Perelman showed the existence of an atlas on $X^*$ compatible with the distance whose transition maps are Euclidean $DC$ functions. Moreover, he also showed that $DC$ functions on Alexandrov spaces (defined in geometric terms involving geodesics) are standard Euclidean $DC$ functions when read in a chart. This notion is also natural if one considers the particular case of closed convex hypersurfaces in Euclidean space where the standard charts are inverse of convex functions. The $DC$ structure is also considered in \cite{kuwae} in connection with properties of the heat kernel on Alexandrov space.

Besides providing a complete proof of the results related to this calculus discussed in \cite{DC}, we also generalize them in some aspects. Note that we choose to present our results in the case of tensors (mainly covariant ones) instead of differential forms as it is in the case in the aforementioned papers. However, it is then a standard algebraic matter to infer from our results the corresponding statements on differential forms (up to considering the orientable double cover of $X^*$ if necessary). More precisely, the calculus we develop allows us to consider general bounded $\GBV$ functions (compared to standard $BV$ functions, the $\GBV$ functions satisfy an additional weak continuity property, see (\ref{geomBV})), thus removing the restriction to bounded $\GBV$ functions with no jump part in their derivative (these functions are called $BV_0$ functions in Perelman's paper)\footnote{We invite the reader to compare our Section~\ref{DC0} with lemma p5 in 
Perelman's paper \cite{DC}, and our Section~\ref{compacond} with Section 4.3 in the same paper} as in \cite{DC,kuwae}. The main result on tensors can be summarized as the existence of a covariant derivative operator on tensors $S$ with $\GBV$ components (thus $DS$ is a measure-valued tensor) that satisfies the same properties as on a smooth Riemannian manifold. We are also able to define the norm of a measure-valued tensor that coincides with the standard norm in the case of function-valued tensor and gives us, for example, an intrinsic notion of total variation $|df|$ for a function of bounded variation $f$ on $X^*$. 
 
Our particular interest in covariant tensors stems from our will to define the Hessian of a distance function $d_p$ as a measure-valued tensor. Indeed, even in the case of a smooth Riemannian manifold, the distributional Hessian $\hes \,d_p$ of $d_p$ is {\it not}, in general, absolutely continuous with respect to the volume measure, see for instance \cite{mantegazza} for a nice and comprehensive discussion of this fact. The extra term in $\hes \,d_p$ is of jump type, namely it is concentrated on the Cut Locus of $p$ which is $(N-1)$-dimensional in general, being $N$ the dimension of the Riemannian manifold. For example, this phenomenon arises on the real projective space, details can be found in \cite{GHL}. Consequently, even on a smooth Riemannian manifold, the Riemannian gradient $\nabla d_p$ does not belong to the class $BV_0$ defined by Perelman in general. Our results allow us to define the Hessian of any distance function $d_p$ as a particular instance of a result for arbitrary $DC$ functions. This generalization requires fine properties of $BV$ functions and a thorough study of the jump part of the Hessian. Technically speaking, the primary issue is that, given two $BV$ functions $f,\,h$ defined on an open set of Euclidean space, the standard chain rule 
$$ \frac{\partial (fh)}{\partial x_i} = f \frac{\partial h}{\partial x_i} +   h \frac{\partial f}{\partial x_i}$$
does {\it not} hold true when $f,\,h$ have non-trivial jump parts in their derivatives (see however Lemma~\ref{lemmaproduct}). However, even when
jump parts in the derivatives are present, it may happen that cancellation properties occur, so that some standard formulas and definitions of tensor calculus retain their
validity (see Proposition~\ref{FormHessian} and  Proposition~\ref{gamma2} as a particular instances of this phenomenon).

As a consequence of this extended $DC$ Calculus, we also prove that the $DC$ Laplacian (namely the trace of Hessian with respect to the Riemannian metric) of $DC$ functions coincides with the weak Laplacian defined through integration by parts (see Proposition~\ref{IBP}). This fact has been used in \cite{LVPet} without an
explicit proof. 

We are also able to prove that for a $DC$ function $f$ and $\GBV$ vector fields $X,\,Y$, the Hessian of $f$ is related to the covariant derivative in the same way as in the smooth case in the sense that
$$ \hes f (X,Y)= D(df(Y))(X)- df(D_XY)$$
(where $D$ stands either for the differential of function or the (measure-valued) covariant derivative of vector field, depending on the context). We also prove that the classical formula
$$ D_{\nabla \psi} \nabla \psi= \frac{1}{2} \nabla (|\nabla \psi|_g ^2)$$
remains valid in this setting.

The formula above is then used to prove that the Hessian we define satisfies an integration by parts formula in the same spirit as in the $\Gamma_2$ calculus 
\cite{Bakry-94}, see Proposition~\ref{gamma2}. This consistency result is particularly relevant in connection with the coordinate-free approach
to calculus in metric measure spaces, which works particularly well under Ricci lower bounds on curvature, still expressed in terms of
the $\Gamma_2$ tensor, see \cite{Gigli}. We know after 
\cite{LVPet} and \cite{AGSAP} that spaces with sectional curvature bounded from below satisfy Ricci lower bounds in this sense, see also \cite{Sturm} 
for a very nice application of these calculus tools to Alexandrov spaces. Nevertheless the connection
between this viewpoint and the one developed in this paper is not yet completely clear, since the calculus developed in \cite{Gigli,AGSAP} allows to handle, 
so to speak, only sections of the tangent bundle defined up to $N$-negligible sets, while the calculus in coordinates allows to handle 
sections of the tangent bundle defined up to $(N-1)$-dimensional sets.

The $DC$ Calculus allows Perelman to improve an important result of Otsu and Shioya in \cite{OS} where the authors prove that the distance on Alexandrov space $X$ derives from a continuous (in a weak sense, see Theorem B in \cite{OS}) Riemannian metric defined almost everywhere on $X^*$. 
Perelman proves that the Riemannian metric components, when read in a chart, are actually functions of local bounded variation. In \cite{ABSurface}, we 
 improve these results in the case of surfaces.

This paper is organized as follows. In the first part, we introduce the function spaces related to the $DC$ Calculus. We then introduce the notion of $DC_0$ Riemannian manifold, which provides a natural setting where the $DC$ Calculus can be defined. In the subsequent part, we study the covariant tensors and their covariant derivatives on a $DC_0$ Riemannian manifold. Due to the low regularity, we use the old-fashioned approach consisting in defining tensors in charts and imposing a compatibility condition (see \cite{SpivakI} for more on this point of view). The antepenultimate section is devoted to the Laplacian and Hessian of a $DC$ function. The penultimate one is devoted to the integration by parts formula involving Hessian. In the last part, we establish that the subset $X^*$ of a finite dimensional Alexandrov space $X$ is indeed a $DC_0$ Riemannian manifold.

\section{Function spaces for $DC$ Calculus}

\subsection{$DC$, $BV$ and $\Cw$ functions}

\begin{defi}[DC functions] Let $\Omega \subset \R^N$ be an open set. A function  $f: \Omega \rightarrow \R$ is said to be a DC function on $\Omega$ if it is locally representable in $\Omega$ as a difference of two semiconcave functions.
\end{defi}

A similar definition will be used for vector-valued maps, arguing componentwise.
Since SC functions are defined by a local property, a simple partition of unity shows that DC functions in an open set can also be globally
represented by the difference of two SC functions. It is also easy to check that the class of DC functions is a vector space. 
In the note \cite{hartman}, Hartman proved that the space of DC functions is indeed an algebra, i.e. stable under multiplication. A more
general statement, that encompasses also this property, is given below.

\begin{thm}[Stability of DC]
Given $f: U \subset \R^M \longrightarrow V \subset \R^N$ and $g: V  \longrightarrow \R$  with $U$ and $V$ two open sets, the
function $g \circ f$ is DC on $U$ if $f$ and $g$ are DC on $U$ and $V$ respectively. As a consequence, the set of 
real-valued DC functions on $U$ is an algebra.  
\end{thm}

It is a classical result that the second distributional derivative of a SC function in an open set $V\subset\R^N$
is a  symmetric matrix-valued Radon measure in $V$
whose positive part is absolutely continuous with respect to the Lebesgue measure $\Leb{N}$, and that SC functions admit 
at almost every point a second order Taylor expansion \cite{Alex39}. In particular, by linearity, the first derivatives of a DC function are functions of locally 
bounded variation. For later use, let us quickly review some basic definitions and results on $BV$ functions on Euclidean spaces.

\begin{defi}[$BV$ functions] Let $V$ be an open set of $\R^N$. A function $f : V \longrightarrow \R$ belongs to $BV(V)$ if
$f\in L^1(V)$ and the derivative in the sense of distributions of $f$ is representable by a vector-valued measure
$$
Df=\bigl(D_1f,\ldots,D_N f)
$$
with finite total variation in $V$. Equivalently, $|Df|(V)<\infty$ and  
$$
\int_V f\frac{\partial\phi}{\partial x_i}\,dx=-\int_V \phi\, dD_if\qquad\forall i=1,\ldots,N,\,\,\,\phi\in Lip_c(V).
$$
Here $Lip _c(V)$ stands for the set of Lipschitz functions with compact support on $V$.\\ 
The function $f$ is said to be in $BV_{loc}(V)$ if it is a $BV$ function on every open set $V'\Subset V$.
 \end{defi}

Again, a similar definition can be given componentwise for maps $f : V \longrightarrow \R^M$
 and in this case $Df$ will be viewed as a $(M\times N)$-matrix of measures $D_if^{(j)}$.
Denoting by $\H^k$, $k>0$, the $k$-dimensional Hausdorff measure (the ambient space is irrelevant and it will not appear in the notation), we 
recall that a set  $E\subset \R^N$ is said to be {\it $\sigma$-finite w.r.t. $\H^k$} if $E$ is the union of an 
increasing sequence of sets with finite $\H^k$ measure. If $k\geq 1$ is an integer, we say that a set $E\subset\R^N$ is 
{\it countably $\H^k$-rectifiable} if there exist countably many closed sets
$C_i\subset\R^k$ and Lipschitz functions $f_i: C_i\longrightarrow\R^N$ such that 
\[ \H^k(E \setminus \bigcup_{i=0}^{\infty} f_i(C_i)) =0.\]

\begin{remark} Throughout this paper, $\langle \cdot,\cdot\rangle$ stands for the standard Euclidean inner product and $\|\cdot\|_2$ for its induced norm.
\end{remark}

 Given a locally integrable function $f$ we say that $x$ is an approximate continuity point of $f$ if there exists
$\tilde f(x)\in\R$ satisfying
\begin{equation}\label{eq:ihp1}
\lim_{\rho \downarrow 0} \fint_{B_\rho(x)} | f(y) - \tilde{f}(x)| \,d\Leb{N}(y)=0,
\end{equation}
where, here and in the sequel, $B_\rho(x)$ denotes the open ball with radius $\rho$ and center $x$ and
$\fint_B f\,d\mu$ denotes the averaged integral, i.e. $\int_B f\,d\mu/\mu(B)$. The complement of the set of approximate continuity
points will be denoted by $S_f$. We shall also use the same notation and concept for vector-valued functions, arguing
componentwise.

\begin{thm}[Decomposition of $Df$]\label{decompo} Let $V\subset\R^N$ be an open set. For all $f\in BV_{loc}(V)$ the following properties hold:
\begin{itemize}
\item[(a)] $|Df|$ vanishes on $\H^{N-1}$-negligible sets.
\item[(b)] $Df$ admits the mutually singular decomposition 
\begin{equation}\label{Bvdecom} 
Df = D^{ac}f+D^{si} f=D^{ac}f+ D^{ju}f + D^{ca} f,
\end{equation}
where  $D^{ac}f$ is the absolutely continuous part w.r.t. Lebesgue measure $\Leb{N}$, $D^{si} f$ is the singular part, which can be further split in 
jump part $D^{ju} f$ and Cantor part $D^{ca} f$, the former concentrated on a set $\sigma$-finite w.r.t. $\H^{N-1}$ and the latter vanishing
on all sets with finite $\H^{N-1}$ measure.
\item[(c)] $D^{ju} f$ is concentrated on the set of approximate jump points $J_f$ and the set $J_f$ is a countably $\H^{N-1}$-rectifiable set. 
Moreover up to a $\H^{N-1}$-negligible set, $J_f$ coincides with the approximate discontinuity set $S_f$ of $f$. 
\item[(d)]  By the very definition of $J_f$, all points $x\in J_f$ one can identify two values $f^\pm(x)$, the so-called one-sided approximate limits, 
such that an approximate continuity 
property holds by averaging along an halfspace passing through $x$.
 More precisely, there exists a Borel function $\nu(x) \in \S^{N-1}$, such that
$$ \lim_{\rho \downarrow 0} \fint_{B_{\rho}^{\pm}(x,\nu(x))} | f(y) - f^{\pm}(x)| \,d\Leb{N}(y)=0,$$
where
$$B_{\rho}^{+}(x,\nu(x)) = \{y \in B_{\rho}(x); \langle y-x,\nu(x)\rangle > 0\}, \,\,
B_{\rho}^{-}(x,\nu(x)) = \{y \in B_{\rho}(x); \langle y-x,\nu(x)\rangle < 0\}.$$
The triple $(\nu(x), f^+(x),f^-(x))$ is uniquely determined up to $\nu(x) \mapsto -\nu(x)$ that maps  $f^+(x)$ to $f^-(x)$ and vice versa. 
The jump part of $Df$ is equal to
$$ D^{ju}f = (f^+(x)-f^-(x))\nu(x) \,\H^{N-1}\res{J_f}.$$
\end{itemize}
\end{thm}

For a proof of this proposition, we refer to the the book \cite[Chapter 3]{AFP} (see also Definition~\ref{def:prec_res} below).

We will also use the following result.

\begin{thm}\label{LipCompo} Let $F: U \rightarrow V$ be a biLipschitz homeomorphism 
between open subsets of $\R^N$ such that, either $\det dF\geq 0$ $\H^N$-a.e. on $U$ or  
$\det dF\leq 0$ $\H^N$-a.e. on $U$. Then, for all $f\in BV(U)$ the function $f\circ F^{-1}$ belongs to $BV(V)$ and 
$$ |D(f \circ F^{-1})| \leq (\mbox{Lip } F)^{N-1} F_{\sharp} (|Df|),$$
where $\cdot_{\sharp}$ stands for the pushforward of a measure (see Section~\ref{sec:3.3}).
\end{thm}
For a proof, we refer to \cite[Theorem 3.16]{AFP}. 

\begin{lemma}\label{lem:constancy}
Let $f\in BV(\Omega)$, with $\Omega\subset\R^N$ open and connected. If $D^{ju} f=0$ and, for any ball $B_r(x)\Subset\Omega$
one has $|f|^{-1}\in L^\infty(B_r(x))$, then the sign of $f$ is constant.
\end{lemma}
\begin{proof} By the connectivity of $\Omega$, it suffices to show that if $B$ is a ball, $f\in BV(B)$, $|f|^{-1}\in L^\infty(B)$
and $D^{ju} f=0$ in $B$ imply that $f$ has constant sign in $B$. Notice that $D^{ju} f=0$ in $B$ implies $\H^{N-1}(B\cap S_f)=0$.
If $f$ has not a constant sign, then both $\{f>0\}$ and $\{f<0\}$ have positive
$\H^N$-measure. By the coarea formula we can find $s,\,t>0$ with $|f|\geq\max\{s,t\}$ in $B$ 
such that $\{f<-s\}$ and $\{f>t\}$ have finite perimeter. By our choice of $s$ and $t$ and by the hypothesis on $|f|^{-1}$, 
$B$ is partitioned up to $\H^N$-negligible sets in two sets of finite perimeter which therefore have in $B$ the same essential boundary (namely the set of
points of density neither 0 nor 1). On the other hand, since \eqref{eq:ihp1} implies that all sets
$$
\{y\in B_\rho(x):\ |f(y)-\tilde f(x)|>\epsilon\}\qquad\epsilon>0
$$
have $0$ density at $x$ whenever $x\notin S_f$, it follows that
any point in the intersection of the essential
boundaries of $\{f<-s\}$ and $\{f>t\}$ belongs to $S_f$. It follows that the essential boundary of these
sets is $\H^{N-1}$-negligible in $B$, and this forces one of the two sets (by the relative isoperimetric inequality for sets
of finite perimeter)
to be $\H^N$-negligible and the other one to coincide, up to $\H^N$-negligible sets, with $B$.
\end{proof}

\begin{defi}[Precise representative of a $BV$ function]\label{def:prec_res}
Let $\Omega\subset\R^N$ open and $f\in BV(\Omega)$. Recall that, by definition, 
at all points $x\in\Omega\setminus S_f$ there exists a number, called 
approximate limit and denoted $\tilde f(x)$, satisfying
$\int_{B_r(x)}|f-\tilde f(x)|\,dy=o(r^N)$.
In addition, Lebesgue's theorem gives that $\tilde f=f$ $\Leb{N}$-a.e. in $\Omega$.\\
Analogously, at all points $x\in J_f$ one can identify two values $f^\pm(x)$, the so-called one-sided approximate limits, where an analogous
property holds by averaging along an halfspace passing through $x$ (see Theorem~\ref{decompo} (d)). The precise representative $f^*$ is defined on
$\Omega\setminus (S_f\setminus J_f)$ by
$$
f^*(x)=
\begin{cases} \tilde f(x) &\text{if $x\in\Omega\setminus S_f$};\cr
\frac 12\bigl(f^+(x)+f^-(x)\bigr) &\text{if $x\in J_f$.}
\end{cases}
$$
and it is left undefined on the $\H^{N-1}$-negligible set $S_f\setminus J_f$ (in particular, $Df =D f^*$). 
In the rest of this paper a $BV$ function $f$ will be always identified with its precise representative.
\end{defi}

We refer to \cite[Corollary 3.80]{AFP} for more on the precise representative. 
Notice that the use of the representative is necessary in several results in this paper. 
For instance, the precise representative $ f^*$ will be used when the chain rule formula \cite[Theorem 3.96]{AFP} is applied in our setting,
see also Lemma~\ref{lemmaproduct}.
 
 \subsection{$\Cw$ functions, geometric $BV$ functions and Radon measures}\label{BVDC}

 \begin{defi}[$\Cw$ and $\Cwo$ functions]\label{def:cw}
Let $\Omega \subset \R^N$ be an open set and $f:\Omega\to\R$. We say that $f\in\Cw (\Omega)$ if
 there exists a set $S$, $\sigma$-finite with respect to
$\H^{N-1}$, such that:
\begin{itemize}
\item[(a)] $f\vert_{\Omega\setminus S}$ is continuous;
\item[(b)] $f\vert_{\Omega\setminus S}$ is locally bounded in $\Omega$, i.e. for all $x\in\Omega$ one has
$\sup_{B_r(x)\setminus S}|f|<\infty$ for some ball $B_r(x)\Subset\Omega$.
\end{itemize}
We say that $f\in\Cwo (\Omega)$ if, in addition, $S$ can be chosen to be $\H^{N-1}$-negligible.

Notice that (a) does not imply (b), unless $S=\emptyset$.
Occasionally we will use the notation $\Cw(\Omega,S)$ or $\Cwo (\Omega,S)$ to emphasize the role of the set $S$ in the defining
properties of the classes $\Cw(\Omega)$ and $\Cwo(\Omega)$.
\end{defi}

It is obvious that if $f\in\Cw(\Omega,S)$, then $S_f\subset S$ and
$f=\tilde{f}=f^*$ on $\Omega\setminus S$. In particular, for $f\in BV_{loc}(\Omega)\cap\Cwo(\Omega)$, 
the precise representative coincides with $f$ $\H^{N-1}$-a.e. on $\Omega$. Most of the results stated
below in the class $\Cw(\Omega)\cap BV(\Omega)$ (for instance Lemma~\ref{lemmaproduct} below) could actually be stated and proved in the class
$BV(\Omega)$ using the precise representative; however the use of the spaces $\Cw(\Omega)$ and
$\Cwo(\Omega)$ is compatible with our geometric applications and it simplifies some technical aspects.

Notice also that the Euclidean gradient $\nabla f$ of a $DC$ function $f : \Omega \subset \R^N \rightarrow \R$ (arbitrarily defined on the set of points where $f$ is not differentiable) belongs to $\Cw(\Omega)$. The following definition, then, is natural.

\begin{defi}[The class $DC_0(\Omega)$]
Let $\Omega\subset\R^N$ be an open set.
$DC_0(\Omega)$ is the collection of all functions $f\in DC(\Omega)$ such that  $\nabla f \in (\Cwo(\Omega))^N$.
Analogously, $DC_0(\Omega,S)$ is the collection of all functions $f\in DC(\Omega)$ such that  $\nabla f \in (\Cwo(\Omega,S))^N$.
\end{defi}

For later use, let us also introduce 
\begin{equation}\label{geomBV}
\GBV(\Omega)  := BV_{loc}(\Omega) \cap \Cw(\Omega)\qquad \mbox{ and } \qquad\GBV_0(\Omega):= BV_{loc}(\Omega) \cap \Cwo(\Omega).
\end{equation}
Analogous definitions can be given for the subclasses $\GBV(\Omega,S)$, $\GBV_0(\Omega,S)$.

The natural Radon measures related to these spaces are introduced in the following definition.

\begin{defi}[$\GM(\Omega)$ and $\GM_0(\Omega)$ Radon measures]
The set of Radon measures on $\Omega$ that vanish on $\H^{N-1}$-negligible sets is denoted by $\GM(\Omega)$, whereas
the set of Radon measures on $\Omega$ that vanish on $\H^{N-1}$-finite sets is denoted by $\GM_{0}(\Omega)$.
\end{defi}

\begin{remark} In our terminology, a Radon measure $\mu$ on $\Omega$ has finite total variation only on compact
subset of $\Omega$. Consequently, if $\mu$ is a signed Radon measure, $\mu(B)$ makes sense only for Borel sets $B$ 
with compact support in $\Omega$. Note that the derivative of $f \in \GBV_0(\Omega)$ has no jump part, therefore according to 
Theorem~\ref{decompo} it belongs to $\GM_0(\Omega)$ .
\end{remark}

%%%%%%%%%%%%%%%%%%%%%%%%%%%%%%%%%%
%%%%%%%%%%%%%%%%%%%%%%%%%%%%%%%%%%
%%%%%%%%%%%%%%%%%%%%%%%%%%%%%%%%%%
%%%%%%%%%%%%%%%%%%%%%%%%%%%%%%%%%%

\section{$DC_0$ Riemannian manifolds}\label{DC0}

\subsection{Definition and properties of $DC_0$ Riemannian manifolds}

We start with the definition of the class of manifolds we are interested in.

\begin{defi}[$DC_0$ Riemannian manifolds]\label{def:DC0}
A countable at infinity $N$ dimensional topological manifold $X^*$ is said to be a $DC_0$ (open) Riemannian manifold if the following two properties hold:
\begin{itemize}
\item[(a)] There exist a singular set $\Si \subset X^*$  and a maximal atlas of  $X^*$
$(U_{\alpha},\phi_{\alpha})_{\alpha \in \Lambda}$ made of charts $\phi_{\alpha} : U_{\alpha} \rightarrow \R^N$ which are homeomorphisms 
onto their image such that for any $\alpha,\beta \in \Lambda$,
$$ \phi_{\alpha}(U_{\alpha} \cap \Si) \mbox{ is } \H^{N-1}\mbox{-negligible},$$
and any transition map 
$$F= \phi_{\beta} \circ \phi_{\alpha}^{-1} : \hat{U}_{\alpha}:=\phi_{\alpha}( U_{\alpha} \cap U_{\beta}) \rightarrow \phi_{\beta}(U_{\alpha} \cap U_{\beta})=: \hat{U}_{\beta}$$
 belongs to $DC_0(\phi_{\alpha}(U_{\alpha}\cap U_{\beta}),\phi_{\alpha}(U_{\alpha}\cap \Si))$. 
\item[(b)] $X^*$ is endowed with a Riemannian metric $g$ defined on $X^*\setminus \Si$ whose components $g_{ij}$, when read in a chart $(U,\phi)$, 
belong to $\GBV_0 (\phi(U),\phi(U\cap \Si))$ and satisfy
\begin{equation}\label{condiRmetric}
\frac{1}{c(x)}\|p\|_2^2\geq \sum_{1\leq i,j \leq N}g_{ij}(x)p_ip_j\geq c(x)\|p\|_2^2\quad\text{for all $p\in\R^N$, for all $x\in\phi(U\setminus\Si)$,}
\end{equation}
with $c$ and $c^{-1}$ locally bounded in $\phi(U)$. 
\end{itemize}
\end{defi}

\begin{remark}[Sign of determinant of transition maps] The determinant $\det dF$ of the transition map is a $BV$ function
with no jump part of derivative. 
Therefore, Lemma~\ref{lem:constancy} and condition (b) in Definition~\ref{def:DC0} give that the
sign of $\det dF$ is ($\H^N$-almost everywhere) constant on the connected components of the domain
of $F$. This allows us the application of Theorem~\ref{LipCompo}, namely the invariance of the $BV$
property under composition with $F^{-1}$.
\end{remark}

\begin{remark} The Riemannian metric in (b) is in particular a covariant 2-tensor. 
This means that the components of $g$ when read in two different charts satisfy a compatibility condition. 
This condition is recalled in Section~\ref{tensors}.
\end{remark}

\begin{remark}
Note that a $DC_0$ Riemannian manifold is in particular a Lipschitz manifold. Therefore, there is a well-defined first order calculus for locally Lipschitz functions. Namely, a function $f$ is said to be locally Lipschitz if $f \circ \phi^{-1}$ is a locally Lipschitz function with respect to the Euclidean norm when read in an arbitrary chart $(U,\phi)$. Throughout this paper, ``locally Lipschitz'' will always refer to this definition. 
\end{remark}

We can extend all the definitions of Section~\ref{BVDC} to functions or Radon measures on a $DC_0$ Riemannian manifold, by requiring that the corresponding definition 
given there is satisfied when the object is read in any chart $(U,\phi)$. Namely, a function $f$ belongs to the space ``$X$'' if and only if $f\circ \phi^{-1}$ belongs to the space ``$X$'' in the Euclidean sense, and a Radon measure $\mu$ belongs to the space ``$X$" if $\phi_{\sharp} \mu$ belongs to the same space in the 
Euclidean sense. The well-posedness of these definitions is guaranteed by the fact that a $DC_0$ transformation map is in particular a biLipschitz homeomorphism,
and the fact that sets which are $\sigma$-finite or negligible w.r.t. $\H^{N-1}$ are invariant under biLipschitz maps.

%%%%%%%%%%%%%%%%%%%%%%%%%%%%%%%%%%
%%%%%%%%%%%%%%%%%%%%%%%%%%%%%%%%%%

\subsection{Properties of the metric}
We end this part with technical results that will be used in the sequel.

\begin{lemma}\label{g-1}
Let $X^*$ be a $DC_0$ Riemannian manifold. We set $G^{-1}(x)= (g^{ij}(x))_{1\leq i,j\leq N}$ the inverse at $x \in \phi(U \setminus \Si)$ of the matrix 
$(g_{ij}(x))_{1\leq i,j\leq N}$ read in a chart $(U,\phi)$. Then, $G^{-1}$ belongs to $\GBV_0(\phi(U))^{N^2}$. 
\end{lemma}

\begin{proof}
Let $\Omega=\phi(U)$ and let $S$ be the $\H^{N-1}$-negligible singular set of $g$, in these local coordinates. 
Using the expression of the inverse of a matrix in terms of the matrix of cofactors and the fact that $\GBV_0$ is an algebra, it suffices to prove that $1/\bar{g}$, with $\bar{g} = \det (g_{ij})$, belongs to $BV_{loc}(\Omega)$ and is locally bounded in $\Omega$, according to
Definition~\ref{def:cw}. Local boundedness of $1/\bar{g}$ follows immediately by condition (b) in Definition~\ref{def:DC0}.
 The fact that 
$1/\bar{g}$ belongs to $BV_{loc}(\Omega)$ follows from the local boundedness and from the chain rule formula for the left composition with a 
Lipschitz function (see for instance \cite[Theorem 3.96]{AFP}), which gives
$$ D \frac{1}{\bar{g}} = -\frac{1}{\bar{g}^2} D\bar{g}$$ 
provided we work with the precise representative (recall also that $\bar{g}$ has no jump part in its derivative since it belongs to $\GBV_0(\Omega)$). 
\end{proof}

For later use, we also mention the following result whose proof is along the same lines as the one above, thus is left to the reader.

\begin{lemma}\label{lemmadiffdet} Let $X^* $ be a $DC_0$  Riemannian manifold and let $G$ be the  
$\GBV_0$-metric read in a chart. 
Then, understanding the derivatives in the sense of distributions, one has
\begin{equation}\label{diffdet}
\frac{\partial}{ \partial x_i} \left( \sqrt{{\rm det\,}G}\right)= \frac{\sqrt{{\rm det\,}G}}{2} \sum_{k,s} g^{ks}\frac{\partial g_{ks}}{\partial x_i}. 
\end{equation}
\end{lemma}

To conclude this part, we notice that as in the smooth case, a Riemannian metric can be defined locally first and then be extended to a global tensor on the $DC_0$ manifold. Indeed, it is easy to check that (\ref{condiRmetric}) is preserved by $DC_0$ transformation map and the lemma below guarantees the existence of $DC_0$ partition of unity.

\begin{lemma}\label{DCpartition} Given a locally finite atlas $(U_i,\phi_i)_{i \in I}$ on $X^*$, there exists a partition of unity $(\psi_i)_{i \in I}$ 
subordinate to $(U_i)_{i \in I}$ where $\psi_i$ are compactly supported Lipschitz $DC_0$ functions in $U_i$. 
\end{lemma} 
 \begin{proof}
$X^*$ is paracompact, thus there exists a locally  finite subcover $(U_i')_{i \in I}$ of $(U_i)_{i \in I}$ with 
$U_i' \Subset U_i$. Since in Euclidean spaces one can always find, given open sets $A\Subset B$, a function
$\phi\in C^\infty_c(B)$ with $0\leq\phi\leq 1$ and $\phi\equiv 1$ in a neighbourhood of $A$, we can pull these maps (with
$A=A_i=\phi_i(U_i')$, $B=B_i=\phi_i(U_i)$) to obtain $DC_0$ functions $\psi_i$ and build out of them the desired partition of
unity. 
\end{proof}

%%%%%%%%%%%%%%%%%%%%%%%%%%%%%%%%%%
%%%%%%%%%%%%%%%%%%%%%%%%%%%%%%%%%%

\subsection{Pull-back of function and measure on a $DC_0$ Riemannian manifold}\label{sec:3.3}

Recall that, for a proper map $F:X\longrightarrow Y$, the push-forward operator $F_\sharp$ maps 
Radon measures in $X$ to Radon measures in $Y$ via the formula $F_\sharp \mu(B)=\mu(F^{-1}(B))$
for any Borel set $B$ with compact support in $Y$. Equivalently, since Borel functions can be approximated
with simple Borel functions, one can characterize $F_\sharp\mu$
via the change of variables formula
$$
\int_Y \varphi \,dF_\sharp\mu=\int_X\varphi\circ F\,d\mu.
$$
In particular, we obtain the useful formula 
\begin{equation}\label{eq:ihp4}
F_\sharp ((k\circ F)\mu)=kF_\sharp\mu\qquad\text{for any locally bounded Borel function $k:Y\longrightarrow\R$,}
\end{equation} 
since (by applying the chance of variables formula with $\varphi=k\chi_B$)
$$
F_\sharp ((k\circ F)\mu)(B)=\int_{F^{-1}(B)} k\circ F\,d\mu=\int_B k\,dF_\sharp\mu
$$
for any Borel set $B$ with compact support.

For geometrical purposes, we need also to introduce the pull-back of measures, in a form that takes into account
the Jacobian determinant of $F$, see \eqref{eq:ihp3} below.

\begin{defi}[Pull-back of functions and measures]\label{pbmf} Let $X^* $ be a $DC_0$ manifold. Let $F: \hat{U} \longrightarrow \hat{V}$ be a $DC_0$ transition map. 
Given $f \in \Cw (\hat{V})$ we define the pull-back of $f$ through $F$ by 
$$ F^*(f)= f \circ F.$$ 
Analogously, given $\mu \in \GM(\hat{V})$, we define the pull-back of $\mu$ through $F$ by
$$ \langle F^*(\mu),\psi\rangle = \int_{\hat{V}} \psi \circ F^{-1} |\det dF^{-1}| \,\mu(dx),$$
where $\psi$ is any compactly supported and bounded Borel function on $\hat{U}$. Equivalently,
 taking the change of variables formula for $F_\sharp$ into account, one can write
\begin{equation} \label{eq:ihp3}
F^*(\mu)= (F^{-1})_{\sharp} (|\det dF^{-1}| \,\mu).
\end{equation}
\end{defi}

 Combining \eqref{eq:ihp4} and \eqref{eq:ihp3} we immediately get a formula ``dual'' to \eqref{eq:ihp4}, namely 
\begin{equation}\label{eq:stimamax3}
F^*(k\mu)= (k\circ F) F^*(\mu)\qquad\text{for any locally bounded Borel function $k$.}
\end{equation}

If $\mu=\rho\H^N$ and $\psi$ is any compactly supported and bounded Borel function, the
change of variable formula of $\H^N$ for biLipschitz homeomorphisms gives 
 \begin{eqnarray*}\label{blurb}
 \int \psi(x) \,dF^*(\mu)(x) & =  &\int \psi (F^{-1}(x)) \,|\det dF^{-1}|\,\rho(x) \, dx  \\
& = &\int \psi(x)\,\rho(F(x)) \, dx.										
\end{eqnarray*}
Thus,  
\begin{equation}\label{eq:consistency}
F^*(\rho\H^N)= (F^*\rho)\H^N
\end{equation}
and the two definitions given above, for functions and measures, are mutually consistent.

\begin{remark}\label{rem1} Note that $x \mapsto |\det d_xF^{-1}|$ belongs to $\GBV_0(\hat{V})$ thus, in particular, it is defined 
$\mu$-almost everywhere whenever $\mu \in \GM(\hat{V})$.
\end{remark}

\subsubsection{Approximation result and geometrical consequences}
In this part we prove a proposition whose corollary will be used many times in the rest of the paper.

\begin{prop}\label{teclemma}
Let $\Omega$ be an open subset of $\R^N$ and $h:\Omega\to\R$ bounded. 
Assume that either $h\in \Cwo (\Omega)$ and $\mu \in \GM (\Omega)$, or $h \in \Cw (\Omega)$ 
and $\mu\in \GM_0 (\Omega)$. Let  $(\mu_{\ep})_{\ep >0}$ be Radon measures with finite total
variation in $\Omega$ absolutely continuous with respect 
to $\H^N$ such that $\mu_{\ep}  \rightarrow \mu $ in the duality with $C_c(\Omega)$.  We further assume that 

\begin{equation} \label{eq:stimamax}
 \limsup_{\ep \downarrow 0} |\mu_{\ep}| (\Omega) \leq |\mu|(\Omega)<\infty.
 \end{equation}
Then $h\mu_{\ep}  \rightarrow h \mu$ in the duality with $C_b(\Omega)$, the class of bounded continuous functions in
$\Omega$.
\end{prop}

\begin{proof}
The leading idea of the proof is that the set of discontinuity points of a suitable restriction of 
$h$ is negligible with respect to $\mu$ and $\mu_{\ep}$.
From \cite[Theorem 2.2]{AFP}, we infer that $|\mu_{\ep}|  \rightarrow |\mu| $ in the duality with $C_b(\Omega)$. 
As a consequence, we clearly have that $\mu_{\ep}^{\pm}  \rightarrow \mu^{\pm} $ in the duality with $C_c(\Omega)$. 
Combining this remark with the analogous decomposition of $h$ into its positive and negative part, 
the boundedness of $h$ reduces the proof to the case when $\mu_\ep \geq 0$ and $h$ is nonnegative and bounded. 

Let us consider the case $h \in \Cw (\Omega,S)$ 
and $\mu\in \GBV_0 (\Omega)$, the proof in the other case is similar.
We then define the lower semi-continuous envelope $h_-:\Omega\to\R$
$$ h_-(x) = \inf\bigl\{\liminf h(x_n);\ x_n \in \Omega\setminus S \mbox{ and } \lim_{n \rightarrow  \infty} x_n=x\bigl\}.$$
We define the upper-semicontinuous $h_+$ envelope analogously.  Now, using that $h_{-}\geq 0$ 
is lower-semicontinuous and $\mu \geq 0$, we get (see for instance \cite[Proposition 1.62]{AFP} for a proof) for any nonnegative continuous function $\psi$
 $$ \liminf_{\ep \downarrow 0}  \int \psi h_{-} \,d\mu_{\ep} \geq \int \psi h_{-} \,d\mu.$$
 Analogously, if $\psi$ is also bounded, we have
 $$ \limsup_{\ep \downarrow 0}  \int \psi h_{+} \,d\mu_{\ep} \leq \int \psi h_{+} \,d\mu.$$
 Now, the continuity assumption on  $h$ yields $h_+=h_-$ on $\Omega\setminus S$, so that
 $$  \int \psi h_{+} \,d\mu_{\ep}= \int \psi h_{-} \,d\mu_{\ep}= \int \psi h \,d\mu_{\ep}$$
 and 
 $$  \int \psi h_{+} \,d\mu= \int \psi h_{-} \,d\mu= \int \psi h \,d\mu.$$
 This proves the convergence of $h\mu_\ep$ to $h\mu$ in the duality with bounded and nonnegative continuous functions.
 The general case can be achieved splitting the test function $\psi$ in positive and negative part. 
 \end{proof}

\begin{cor}\label{corteclem}
Let $F: \hat{U} \longrightarrow \hat{V}$ be a $DC_0$ transition map. 
Assuming that $h\in\GBV(\hat{V})$, the following chain rule in the sense of Radon measures holds:
\begin{equation}\label{compoder}
\frac{\partial}{\partial x_i} (h \circ F) = \sum_{s=1}^N \frac{\partial F_s}{\partial x_i} \, F^*\left(\frac{\partial h}{ \partial y_s}\right)
\end{equation}
for any $i \in \{1,\cdots,N\}$ and $F=(F_1, \cdots, F_N)$. 
 \end{cor}

 \begin{proof} 
We set $h_{\ep}= h * \rho_{\ep}$ with $\rho_{\ep}$ a family of mollifiers. Then, $h_{\ep}$ and $h_{\ep} \circ F$ are locally Lipschitz functions. Thus, (\ref{compoder}) holds with $h_{\ep}$ instead of $h$. Since $h_\ep\circ F\to h\circ F$ in $L^1_{loc}(\hat{U})$, one has
$$\frac{\partial (h_{\ep}\circ F)}{\partial x_i}  \rightarrow \frac{\partial (h\circ F)}{\partial x_i}\qquad i=1,\cdots,N$$
in the duality with $C^\infty_c(\hat{U})$, and then in the duality with $C_c(\hat{U})$. On the other hand, using $ F^*(h\mu)= (h\circ F)\, F^*(\mu)$, we can rewrite the right-hand side of (\ref{compoder}) as 
$$ \sum_{s=1}^N \frac{\partial F_s}{\partial x_i} \; F^*\left(\frac{\partial h}{ \partial y_s}\right)=  
\sum_{s=1}^N F^*\left(\frac{\partial F_s}{\partial x_i}  \circ F^{-1} \frac{\partial h}{ \partial y_s}\right).$$
Now, $\mu_{\ep}= \partial h_{\ep} /\partial y_s \H^N$ satisfy the hypotheses in Proposition~\ref{teclemma} (a proof is given in \cite[Theorem 2.2]{AFP} for instance), 
$  \partial h/\partial y_s \in \GM(\hat{V})$, $|\det dF^{-1}|\, \partial F_s/\partial x_i  \circ F^{-1} \in \Cw_{,0}(\hat{V})$ therefore Proposition \ref{teclemma} yields
$$|\det dF^{-1}| \,\frac{\partial F_s}{\partial x_i}  \circ F^{-1}  \; \frac{\partial h_{\ep}}{\partial y_s} \H^N\rightarrow
 |\det dF^{-1}| \,\frac{\partial F_s}{\partial x_i}  \circ F^{-1}  \; \frac{\partial h}{\partial y_s}$$
in the duality with $C_c(\hat{U})$. This implies, using $F^*(\mu)= F^{-1}_{\sharp}(|\det dF^{-1}|\mu)$,
$$ F^* (\frac{\partial F_s}{\partial x_i}  \circ F^{-1}  \;\frac{\partial h_{\ep}}{\partial y_s} \H^N) \rightarrow
F^* (\frac{\partial F_s}{\partial x_i}  \circ F^{-1}  \; \frac{\partial h}{\partial y_s})$$
in the duality with $C_c(\hat{U})$.
By combining this convergence with \eqref{eq:stimamax3},  we get (\ref{compoder}).  
 \end{proof}

\subsubsection{Measure induced by a system of measures}

\begin{defi}[System of Radon measures]\label{def:sysrado}
Let $X^*$ be a $DC_0$ Riemannian manifold with charts $(U_{\alpha},\phi_{\alpha})_{\alpha \in \Lambda}$. A family of Radon measures $(\mu_{\alpha})_{\alpha \in \Lambda}$, where $\mu_{\alpha}$ is a Radon measure on $\phi_{\alpha}(U_{\alpha})$  
 is called
a system of Radon measures if it further satisfies the compatibility condition $F^*(\mu_\beta)=\mu_\alpha$ in $\phi_\alpha(U_\alpha\cap U_\beta)$.
\end{defi}
Note that the above condition is not equivalent to $F_{\sharp}(\mu_\alpha)=\mu_\beta$. However, we have the following result.

\begin{lemma}\label{starpush}Let $F : \hat{U} \longrightarrow \hat{V}$ be a $DC_0$ transition map. 
Suppose we are given $\tilde{\mu} \in \GM(\hat{V})$ and set $\mu=F^*(\tilde\mu)\in \GM(\hat{U})$. Then, 
$$ F_{\sharp} (\sqrt{\det G}\mu)=\sqrt{\det\tilde G}\,\tilde{\mu},$$
where $G$ and $\tilde G$ are the metrics in the respective coordinate systems.
 \end{lemma} 
 \begin{proof}
 Let $\psi$ be a bounded Borel function with compact support. In this proof, we will need the fact that
 $$ \begin{array}{ccl}
G & = &  ^t\big( dF\big) \big(\tilde{G} \circ F \big) \left(dF\right)\\
&\mbox{and}&\\
G^{-1} & = &  \left( d_{F(\cdot)} F^{-1}\right) \big( \tilde{G}^{-1} \circ F\big) ^t\left(d_{F(\cdot) } F^{-1}\right).
\end{array}$$
where $^tM$ stands for matrix transposition. These relations are nothing but a reformulation in terms of matrices of the compatibility condition applied to the covariant $2$-tensor $g$ in different charts. This compatibility condition is recalled in Definition~\ref{compacond}. We infer from the equalities above that $ \sqrt{\det \tilde{G}} = \sqrt{\det G} \circ F^{-1} |\det dF^{-1}|$. Thus, we get
\begin{align*}
\langle F^{-1}_{\sharp} (\sqrt{\det\tilde{G}} \,\tilde{\mu}), \psi\rangle &= \int \psi \circ F^{-1}(y)  \sqrt{\det\tilde{G}}(y) \,\tilde{\mu}(dy) \\
					&=\int \psi \circ F^{-1}(y) \sqrt{\det G} \circ F^{-1}(y) |\det dF^{-1}|(y)\, \tilde{\mu}(dy) \\
			               &= \langle \sqrt{\det G}F^*(\tilde{\mu}),\psi\rangle,
\end{align*}
where we used the identity $F^*(h\mu)=(h\circ F)F^*(\mu)$ to get the last equality. 
\end{proof}
For later use, let us also point out that combining the fact that $F$ is a homeomorphism, $F^*(\tilde\mu)= F^{-1}_{\sharp}(|\det dF^{-1}|\tilde \mu)$, and $|\det dF^{-1}|\geq 0$ yields
\begin{equation}\label{TV}
|\mu|=F^*(|\tilde\mu|)
\end{equation}
by uniqueness of the Radon-Nikod\'ym decomposition (see \cite[Corollary 1.29]{AFP} for a precise statement).

\begin{lemma}\label{MeasFSyst} Let $X^*$ be a $DC_0$ Riemannian manifold with charts $(U_{\alpha},\phi_{\alpha})_{\alpha \in \Lambda}$. Then, any system of Radon measures $(\mu_{\alpha})_{\alpha \in \Lambda}$ with $\mu_{\alpha} \in \GM(\phi_{\alpha}(U_{\alpha}))$ for all $\alpha \in \Lambda$, induces a Radon measure 
$\mu \in \GM(X^*)$, characterized by
\begin{equation}\label{eq:eccocosae}
\mu(B)=\int_{\phi_\alpha(B)}\sqrt{\det G_\alpha(x)}\,d\mu_{\alpha} (x)
\qquad\text{for any Borel set $B\subset U_\alpha$.}
\end{equation}
The total variation of this measure is then given by
\begin{equation}\label{eq:eccocosae2}
 |\mu|_g(B)=\int_{\phi_\alpha(B)}\sqrt{\det G_\alpha(x)}\,d|\mu_{\alpha}| (x)
\qquad\text{for any Borel set $B\subset U_\alpha$.}
\end{equation}
Morever, if $\mu_{\alpha} \in \GM_0(\phi_{\alpha}(U_{\alpha}))$ for all $\alpha\in\Lambda$, then $\mu\in \GM_0(X^*)$.
\end{lemma}

\begin{proof}
The proof is based on a standard argument involving partitions of unity. Recall that, according to Lemma~\ref{DCpartition}, there exists a locally finite partition of unity $(\psi_{\alpha})_{\alpha \in \Lambda}$ subordinate to $(U_{\alpha})_{\alpha \in \Lambda}$ where $\psi_{\alpha}$ are compactly supported Lipschitz $DC_0$ functions. Then, given any bounded Borel function $\chi$ with compact support, we set 
$$\int_{X^*} \chi \, d\mu:= \sum_{\alpha \in \Lambda} 
\int_{\phi_{\alpha}(U_{\alpha})} (\psi_{\alpha} \chi) \circ \phi_{\alpha}^{-1}(x) \,\sqrt{\det G_\alpha(x)}\,d\mu_{\alpha} (x).$$
Using Lemma \ref{starpush} and (\ref{TV}), it is a standard fact that the measure is well-defined (i.e. that it does not depend on the partition of unity) and
that it satisfies \eqref{eq:eccocosae} and \eqref{eq:eccocosae2}. We refer to \cite{GHL} or to \cite{SpivakI} for more details. 
Moreover, it is immediately seen that $\mu \in \GM_0(X^*)$ whenever all $\mu_{\alpha} \in \GM_0(\phi_{\alpha}(U_{\alpha}))$.
\end{proof}

%%%%%%%%%%%%%%%%%%%%%%%%%%%%%%%%%%
%%%%%%%%%%%%%%%%%%%%%%%%%%%%%%%%%%
%%%%%%%%%%%%%%%%%%%%%%%%%%%%%%%%%%%%%%%%%%%%%%%%%%%%%%%%
%%%%%%%%%%%%%%%%%%%%%%%%%%%%%%%%%%%%%%%%%%%%%%%%%%%%%%%%

\section{Tensors on $DC_0$ Riemannian manifolds}\label{tensors}

Our goal in this section is to prove that there is a well-defined tensor calculus on $DC_0$ Riemannian manifolds, including covariant derivative of tensors, provided that the components of the tensor belong to $\GBV$. We restrict our attention to covariant tensors since we are mainly interested in defining the Hessian of a $DC$ function on a $DC_0$  Riemannian manifold (this will be done in the next section). However, we will also need to consider $\GBV$ vector fields as a simple instance of contravariant tensors in order to evaluate our tensor fields, this is done in a dedicated part. It is simple, on the basis of these considerations, to extend our arguments to general tensors; the details are left to the reader.  

\subsection{Definition of tensors on $X^*$}
\begin{defi}[Covariant tensors on an open set $V\subset \R^N$]
A covariant $p$-tensor $S$ on $V$ with $\GBV$ (resp. $\GM$) components is by definition,
$$ S = \sum_{j_1, \cdots,j_p} S_{j_1\cdots j_p} dy^{ j_1}\otimes\cdots\otimes dy^{ j_p}$$
where the multi-indices $(j_1, \cdots, j_p)$ run into $\{1,\cdots,N\}^p$ and $S_{j_1\cdots j_p} \in \GBV(\hat{V})$ (resp. $\GM(\hat{V})$).
\end{defi}

Now, we define the pull-back of  a covariant tensor with components in $\GBV(\hat{V})$ or in $\GM(\hat{V})$, through a transition map.

\begin{defi}[Pull-back of a tensor]
Given $F: \hat{U} \longrightarrow \hat{V}$ a $DC_0$ transition map and $S$ a covariant $p$-tensor on $\hat{V}$ 
with components in $\GBV(\hat{V})$ (resp. $\GM(\hat{V})$), we define a covariant $p$-tensor $F^*S$ on $\hat{U}$ by the formula
\begin{equation}\label{F*S}
F^*S= \sum_{i_1, \cdots, i_p} \sum_{j_1,\cdots j_p}  \frac{\partial F_{j_1}}{\partial x_{i_1}} \cdots \frac{\partial F_{j_p}}{\partial x_{i_p}} F^*(S_{j_1\cdots j_p})\, dx^{i_1} \otimes \cdots \otimes dx^{i_p}.
\end{equation}
\end{defi}

\begin{remark} $F^*S$ has $\GBV$ components according to Theorem \ref{LipCompo}. According to the fact that $\GBV_0(\hat{U})$ is an algebra and  Remark~\ref{rem1}, the components of $F^*S$ belong to $\GBV_0(\hat{U})$ (resp. $\GM_0(\hat{U})$) whenever the components of $S$ belong to $\GBV_0(\hat{V})$ (resp. $\GM_0(\hat{V})$). 
\end{remark}

\begin{defi}[Covariant tensors on a $DC_0$ manifold $X^*$]\label{compacond}
A covariant $p$-tensor $S$ on an open subset $\Omega \subset X^*$ with $\GBV$ (resp. $\GM$) components is, by definition, a family 
$$S = \{S_{\alpha}\}_{\alpha \in \Lambda}$$
where for each $\alpha$ such that $U_{\alpha} \cap \Omega\neq \emptyset$, $S_{\alpha}$ is a covariant $p$-tensor on 
$\phi_{\alpha}(U_{\alpha} \cap \Omega)$ with $\GBV(\phi_{\alpha}(U_{\alpha} \cap \Omega))$ (resp. $\GM(\phi_{\alpha}(U_{\alpha}\cap \Omega))$) 
components. The family of tensors is required to satisfy the compatibility condition
$$ F^*S_2= ( \phi_2 \circ \phi_1^{-1})^*S_2= S_1\qquad\text{{ in $\phi_1(U_1\cap U_2\cap\Omega)$,}} 
$$ 
as Radon measures or for $\H^N$-a.e points, depending on the regularity of the components, and for any pair of charts $(U_1,\phi_1)$, $(U_2,\phi_2)$ such that $U_1 \cap U_2\cap \Omega \neq \emptyset$.
 
We define similarly covariant $p$-tensors $S$ on an open subset $\Omega \subset X^*$ with $\GBV_0$ (resp. $\GM_0$) components.
 \end{defi}

 %%%%%%%%%%%%%%%%%%%%%%%%%%%%%%%%%%
%%%%%%%%%%%%%%%%%%%%%%%%%%%%%%%%%%
 
 \subsection{Covariant derivative of tensors on $X^*$}
Our next task is to define the covariant derivative $DS$ of a $\GBV$ tensor $S$  on $X^*$. We proceed as above, by defining first $DS$ in a chart and then verifying the compatibility condition.

\subsubsection{Local definition of covariant derivative}
To this aim, we have to introduce the Christoffel symbols.
 
 \begin{defi}[Christoffel symbols]Let $X^* $ be a $DC_0$ manifold with Riemannian metric $g$. The Christoffel symbols are then defined as a collection of Radon measures $\{(\Gamma^{(\alpha)})^k_{ij}\}_{\alpha \in \Lambda}$, where 
  \begin{equation}\label{Chris} \Gamma_{ij}^k = \frac{1}{2} \,\sum_{l=1}^N g^{kl}\left( \frac{\partial g_{li}}{\partial x_j} + \frac{\partial g_{lj}}{\partial x_i} - \frac{\partial g_{ij} }{\partial x_l}\right)
\end{equation}
belongs to $\GM_0$ with singular set $\Si$, the superscript $^{(\alpha)}$ is omitted for simplicity, and 
$g_{ij}$ are the components of the metric in the chart $(U_{\alpha}, \phi_{\alpha})$.
 \end{defi}
\begin{remark} In the appendix, it is proved that the Christoffel symbols read in different charts satisfy the same compatibility relation as on a smooth Riemannian manifold.
\end{remark}

With this definition in hands, we can now introduce the notion of covariant derivative.
 
 \begin{defi}[Covariant derivative of a $\GBV$ tensor in a chart $\hat{V}\subset \R^N$]
 Given a covariant $p$-tensor $S$ on $\hat{V}= \phi(V)$ with $\GBV$ components, we define a covariant $(p+1)$-tensor $DS$ by the formula
$$ DS = \sum_{i_0, \cdots, i_{p}} (DS)_{i_0\cdots i_{p}} dx^{i_0} \otimes \cdots \otimes dx^{i_{p}},$$
where 
\begin{equation}\label{CoDe} (DS)_{i_0\cdots i_{p}} = 
\frac{\partial S_{i_1\cdots i_{p}}}{\partial x_{i_0}}   - \sum_{j=1}^p \sum_{m=1}^N  S_{i_1\cdots i_{j-1}mi_{j+1}\cdots i_p}\displaystyle \Gamma_{i_0i_j}^m.
\end{equation} 
The components of $DS$ belong to $\GM(\hat{V})$.
 \end{defi}
 
 In order to infer from the above definition a well-defined notion of covariant derivative of tensor on any open subset of $X^*$, it remains to prove the compatibility formula
 $$ F^*(DS) = D(F^*S).$$
 
 This is the goal of the next section.

\subsubsection{Covariant derivative of $\GBV$ tensors on $X^*$}

We start with the following fact.

\begin{lemma}\label{lemmaproduct} Let $\Omega \subset \R^N$ be an open set and $f,\,h \in \GBV(\Omega)$. Assume that at least one of the two functions $f,\,h$ belongs to $\GBV_0(\Omega)$. Then, for any $i \in \{1,\cdots,N\}$, it holds
$$ \frac{\partial (fh)}{\partial x_i} = f \frac{\partial h}{\partial x_i} + h \frac{\partial f }{\partial x_i}.$$ 
\end{lemma}

\begin{proof}
For mere $\GBV(\Omega)$ functions, the above formula is not true in general because of the jump part of the derivatives (for instance
if $f=h$ is a characteristic function). Under the assumption of the lemma, at most one of the two functions has a jump part in its derivative. 
Then, the standard chain rule formula holds as proved in \cite[Theorem 3.96]{AFP} for instance. 
\end{proof}

\begin{prop}\label{compa} Let $F: \hat{U} \rightarrow \hat{V}$ be a $DC_0$ transition map. Let $S$ be a covariant $p$-tensor with components in $\GBV(\hat{V})$. Then, the following equality of Radon measures holds:
\begin{equation}\label{Comp}
 F^*(DS) = D(F^*S).
 \end{equation}
\end{prop}

\begin{proof}
By linearity of $F^*$, and up to a permutation of the coordinates, 
it is sufficient to prove (\ref{Comp}) in the case when $S = S_{1\cdots p} \, dy^1\otimes \cdots \otimes dy^{p}$.  We start with $F^*(DS)$. In the following, we set $\i = (i_0,\cdots, i_p)$ and $\j= (j_0,\cdots, j_p)$ for arbitrary multi-indices in $\{1, \cdots, N\}^{p+1}$.
By definition of covariant derivative,
$$ (DS)_{\j}= \frac{\partial S_{j_1 \cdots j_p}}{ \partial y_{j_0}} - 
\sum_{k=1}^p \sum_{m=1}^N S_{j_1\cdots j_{k-1} m j_{k+1} \cdots j_p} \tilde{\Gamma}_{j_0j_k}^m,
$$
where the $\tilde{\Gamma}_{j_0j_k}^m$ stand for the Christoffel symbols in the chart relative to $\hat{V}$. This yields
$$ \begin{array}{rcl}
 F^*(DS) 
 			& = & \sum_{\j} \sum_{\i} \frac{\partial F_{j_0}}{\partial x_{i_0}} \cdots \frac{\partial F_{j_p}}{\partial x_{i_p}} F^*( (DS)_{\j} ) \,dx^{i_0}\otimes \cdots \otimes dx^{i_{p}}.
\end{array}$$
% & = & \sum_{\j} F^*((DS)_{\j})\, dF^{j_0} \, \otimes \cdots \otimes dF^{j_{p}}\\
Therefore, the components of $F^*(DS)$ satisfy
\begin{eqnarray}\label{FDS}
\big(  F^*(DS)\big)_{\i} &=& \sum_{\j} \frac{\partial F_{j_0}}{\partial x_{i_0}}\cdots \frac{\partial F_{j_{p}}}{\partial x_{i_{p}}}\Big( F^*\Big( \frac{\partial S_{j_1\cdots j_{p}}}{\partial y_{j_0}}\Big)\Big) \nonumber \\
			&  &    - \sum_{\j} \sum_{k=1}^{p} \sum_{m=1}^N  \frac{\partial F_{j_0}}{\partial x_{i_0}}\cdots \frac{\partial F_{j_{p}}}{\partial x_{i_{p}}} S_{j_1\cdots j_{k-1}mj_{k+1}\cdots j_{p}} \circ F\, F^*\big(  \tilde{\Gamma}_{j_0j_k }^m\big). \nonumber  \\
\hskip -45pt \mbox{Thus, using that } &   S= & S_{1\cdots p} \, dy^1 \otimes \cdots \otimes dy^{p},  \nonumber \\			
		\big(  F^*(DS)\big)_{\i}	&=& \sum_{j_0} 		\frac{\partial F_{j_0}}{\partial x_{i_0}} \frac{\partial F_{1}}{\partial x_{i_1}}\cdots \frac{\partial F_{p}}{\partial x_{i_{p}}}\Big( F^*\Big( \frac{\partial S_{1 \cdots p}}{\partial y_{j_0}}\Big)\Big)  \\	
			&  &   - \sum_{k=1}^p\sum_{j_0, j_k}   \Big(\frac{\partial F_1}{ \partial x_{i_1}} \cdots \hat{\frac{\partial F_k}{ \partial x_{i_k}}} \cdots \frac{\partial F_p}{ \partial x_{i_p}}\Big)  \frac{\partial F_{j_0}}{ \partial x_{i_0}}\frac{\partial F_{j_k}}{ \partial x_{i_k}}  S_{1\cdots p} \circ F\,  F^*\big(  \tilde{\Gamma}_{j_0j_k}^k\big). \nonumber 
\end{eqnarray}
Now, we compute the components of $D(F^*S)$. Starting from
$$ F^* S = \sum_{(i_1,\cdots, i_p)} \frac{\partial F_1}{\partial x_{i_1}} \cdots \frac{\partial F_p}{\partial x_{i_p}} S_{1\cdots p} \circ F dx^{i_1} \otimes\cdots \otimes d^{i_p},$$
we get,
 \begin{eqnarray}\label{DFS1}
\Big( D(F^*S)\Big)_{\i} & =& \frac{\partial}{\partial x_{i_0}} \big( (F^*S)_{i_1 \cdots i_{p}}\big) \nonumber \\
& &   - \sum_{k=1}^{p} \sum_{m=1}^N (F^*S)_{i_1 \cdots i_{k-1}mi_{k+1} \cdots i_{p}}\Gamma_{i_0i_{k}}^m \nonumber \\
                    \Big( D(F^*S)\Big)_{\i}              & = & \frac{\partial F_1}{\partial x_{i_1}} \cdots \frac{\partial F_p}{\partial x_{i_{p}}} \frac{\partial}{\partial x_{i_0}} \big( S_{1 \cdots p} \circ F\big) + \big( S_{1 \cdots p} \circ F\big) \frac{\partial}{\partial x_{i_0}} \Big( \frac{\partial F_1}{\partial x_{i_1}} \cdots \frac{\partial F_p}{\partial x_{i_{p}}}\Big) \\
                                  &     &   - \sum_{k=1}^{p} \sum_{m=1}^N \frac{\partial F_1}{\partial x_{i_1}} \cdots \frac{\partial F_{k-1}}{\partial x_{i_{k-1}}}\frac{\partial F_{k}}{\partial x_{m}} \cdots \frac{\partial F_p}{\partial x_{i_{p}}} \,S_{1\cdots p} \circ F\,\Gamma_{i_0i_{k}}^m.\nonumber
\end{eqnarray}
Now according to (\ref{compoder}) in Corollary~\ref{corteclem}, we have
\begin{equation}\label{DFS2} \frac{\partial}{\partial x_{i_0}} \big( S_{1 \cdots p} \circ F\big) = \sum_{u=1}^N \frac{\partial F_u}{\partial x_{i_0}} F^* \Big( \frac{\partial S_{1 \cdots p}}{\partial y_u}\Big).
\end{equation}

Therefore, according to (\ref{FDS}), (\ref{DFS1}), and (\ref{DFS2}), $F^*(DS)=D(F^*S)$ if and only if for any multi-index $\i$ we have
\begin{multline}\label{DFS3}
- \sum_{k=1}^p\sum_{j_0, j_k}   \Big(\frac{\partial F_1}{ \partial x_{i_1}} \cdots \hat{\frac{\partial F_k}{ \partial x_{i_k}}} \cdots \frac{\partial F_p}{ \partial x_{i_p}}\Big)  \frac{\partial F_{j_0}}{ \partial x_{i_0}}\frac{\partial F_{j_k}}{ \partial x_{i_k}}  S_{1\cdots p} \circ F\,  F^*\big(  \tilde{\Gamma}_{j_0j_k}^k\big) \\
= \big( S_{1 \cdots p} \circ F\big) \frac{\partial}{\partial x_{i_0}} \Big( \frac{\partial F_1}{\partial x_{i_1}} \cdots \frac{\partial F_p}{\partial x_{i_{p}}}\Big)
\\ - \sum_{k=1}^{p} \sum_{m=1}^N \frac{\partial F_1}{\partial x_{i_1}} \cdots \frac{\partial F_{k-1}}{\partial x_{i_{k-1}}}\frac{\partial F_{k}}{\partial x_{m}} \cdots \frac{\partial F_p}{\partial x_{i_{p}}} \,S_{1\cdots p} \circ F\,\Gamma_{i_0i_{k}}^m.
\end{multline}
 This equality can be obtained from the following transformation law of the Christoffel symbols (whose proof in our setting is given in the appendix):
 \begin{equation}\label{TL1} \Gamma_{i_0i_k}^m = \sum_{\theta} \frac{\partial F_m ^{-1}}{\partial y_{\theta}} \circ F \left( \sum_{u,v } \frac{\partial F_u}{\partial x_{i_0}} \frac{\partial F_v}{\partial x_{i_k}} F^* (\tilde{\Gamma}_{uv}^{\theta})\right) \\
+ \sum_{\theta} \frac{\partial  F_m^{-1}}{\partial y_{\theta}} \circ F \frac{\partial ^2 F_{\theta}}{\partial x_{i_0} \partial x_{i_k}}. 
\end{equation}
Indeed, from the above formula we infer
$$\begin{array}{ll}
\displaystyle \sum_{k=1}^{p} \sum_{m=1}^N  & \displaystyle \frac{\partial F_1}{\partial x_{i_1}} \cdots \frac{\partial F_{k-1}}{\partial x_{i_{k-1}}}\frac{\partial F_{k}}{\partial x_{m}} \cdots \frac{\partial F_p}{\partial x_{i_{p}}} \,S_{1\cdots p} \circ F\,\Gamma_{i_0i_{k}}^m  \\
 & \begin{array}{ll}
     \displaystyle = \sum_{k=1}^{p} \sum_{\theta=1}^N &\displaystyle \bigg[\frac{\partial F_1}{\partial x_{i_1}}  \cdots \frac{\partial F_{k-1}}{\partial x_{i_{k-1}}}\hat{\frac{\partial F_{k}}{\partial x_{i_k}}} \cdots \frac{\partial F_p}{\partial x_{i_{p}}} S_{1\cdots p} \circ F \Big( \sum_{m=1}^N \frac{\partial F_k}{\partial x_m} \frac{\partial F_m^{-1}}{\partial y_{\theta} } \circ F \Big) \\
      & \displaystyle \times \Big(\sum_{u,v } \frac{\partial F_u}{\partial x_{i_0}} \frac{\partial F_v}{\partial x_{i_k}}  F^* (\tilde{\Gamma}_{uv}^{\theta}) + \frac{\partial ^2 F_{\theta}}{\partial x_{i_0} \partial x_{i_k}} \Big)\bigg]
       \end{array}   \\
&   \displaystyle
 = \sum_{k=1}^{p} \sum_{j_0,j_k} \bigg[\frac{\partial F_1}{\partial x_{i_1}}  \cdots \frac{\partial F_{k-1}}{\partial x_{i_{k-1}}}\hat{\frac{\partial F_{k}}{\partial x_{i_k}}} \cdots \frac{\partial F_p}{\partial x_{i_{p}}}  \frac{\partial F_{j_0}}{\partial x_{i_0}} \frac{\partial F_{j_k}}{\partial x_{i_k}}   S_{1\cdots p} \circ F  F^* (\tilde{\Gamma}_{j_0j_k}^{k})\bigg]\\
 & \displaystyle \;\;\;\;\;\;\;\;\;\;\;\; + S_{1 \cdots p} \circ F \, \frac{\partial}{\partial x_{i_0}} \Big( \frac{\partial F_1}{\partial x_{i_1}} \cdots \frac{\partial F_p}{\partial x_{i_p}} \Big)
\end{array}$$
where we use that $\sum_{m=1}^N \frac{\partial F_k}{\partial x_m} \frac{\partial F_m^{-1}}{\partial y_{\theta} } \circ F= \delta_{k\theta}$ in $\GBV_0(\hat{U})$ to get the first equality, and Lemma \ref{lemmaproduct} to get the last one. Inserting the last equality into (\ref{DFS3}) gives the result.
\end{proof}

%%%%%%%%%%%%%%%%%%%%%%%%%%%%%%%%%%
%%%%%%%%%%%%%%%%%%%%%%%%%%%%%%%%%%

\subsection{Vector fields}

Let us recall that the notion of $DC_0$ Riemannian manifold is a particular instance of Lipschitz manifold. Therefore, vector fields are well-defined objects on such a manifold. In this part, we briefly give the definition of $\Cw$, $\GBV$, and $\GM$ vector fields on a $DC_0$  Riemannian manifold $X^*$ ($\GBV$ vector fields are well-defined thanks to Theorem~\ref{LipCompo} and the fact that $\GBV$ is an algebra). After that, we verify that we can evaluate $\GM$ tensors by means of $\GBV$ or $\Cwo$ vector fields and get an intrinsic system of Radon measures.

\subsubsection{$\Cw$, $\GBV$, and $\GM$ vector fields}

\begin{defi}[$\Cw$, $\GBV$, and $\GM$ vector fields]\label{defiVF}
Let $\Omega \subset X^*$ be an open set. A $\Cw$ (resp. $\GBV$, $\GM$) vector field on $\Omega$ is the datum of a system of vector fields. 
Namely, for any chart $(U,\phi)$ such that  $U\cap \Omega \neq \emptyset$, we are given a vector field 
$$
X_\alpha=(X_\alpha^1,\cdots,X_\alpha^N): \phi_{\alpha}(U_{\alpha}\cap \Omega) \rightarrow \R^N
$$ 
with components in  $\Cw(\phi_{\alpha}(U_{\alpha}\cap \Omega))$ (resp. $\GBV(\phi_{\alpha}(U_{\alpha}\cap \Omega))$, $\GM(\phi_{\alpha}(U_{\alpha}\cap \Omega))$). Moreover, we assume that for any pair of charts 
$(\phi_{\alpha},U_{\alpha})$, $(\phi_{\beta},U_{\beta})$ such that $U_{\alpha} \cap U_{\beta} \cap \Omega \neq \emptyset$, 
the following compatibility condition holds
\begin{equation}\label{contra}
 F^*(X_{\beta}):= \sum_{j=1}^N \frac{\partial F_i^{-1}}{\partial y_j} F^*(X_{\beta }^j)\frac{\partial }{\partial x_i} = X_{\alpha} \qquad \H^N-a.e.
 \end{equation}
\end{defi}

\begin{remark} Our definition of vector field is a bit non-standard. Here, we identify vector field with contravariant $1$-tensor. However, using that $F$ is a biLipschitz homeomorphism differentiable at any point out of the range of $\Si$, it is easily seen that (\ref{contra}) is equivalent to 
\begin{equation}\label{standvf}
\sum_{i=1}^N \frac{\partial F_j}{\partial x_i} X_{\alpha}^i = F^*(X_{\beta}^j) \qquad \H^N-a.e.
\end{equation}
for all $j \in \{1,\cdots,N\}$. In the sequel, we freely use both conditions. 
\end{remark}

\begin{remark}[Existence of local orthonormal frame made of $\Cwo$ vector fields]\label{gramschmidt}
Let $(\frac{\partial}{\partial x_i})_{1\leq i\leq N}$ be the constant vector fields induced by a coordinate system. These vectors obviously belong to $\Cwo(\Si)$. By applying the Gram-Schmidt process to the $(\frac{\partial}{\partial x_i})$ with respect to $g(x)$ at each point $x$ in the domain of the chart where $g$ is defined ({\it i.e} out of the image of $\Si$), we get a local orthornormal frame $(E_i)$ where the $(E_i)$ are continuous out of the image of $\Si$ through the chart (indeed the metric components satisfy this continuity property). Since the $(E_i)$ are unit vectors with respect to $g$, the property (\ref{condiRmetric}) on $g$ yields that the $(E_i)$ are locally bounded out of a $\H^N$-negligible set. Combining this with the continuity property we just mentioned shows the $(E_i)$ are $\Cwo$ vector fields with singular set the image of $\Si$.
 \end{remark}
  
\begin{lemma}\label{prVfield}Let $\Omega \subset X^*$ be an open set and let $X$ be a $\GBV$ vector field on $\Omega$. Then, given two charts $(\phi_{\alpha},U_{\alpha})$, $(\phi_{\beta},U_{\beta})$ such that $U_{\alpha} \cap U_{\beta} \cap \Omega \neq \emptyset$, if $X_{\alpha}$ is a precise representative of $X$ read in the chart $(\phi_{\alpha},U_{\alpha})$, then
$d F_x(X_\alpha(x))$ is a precise representative of $X_{\beta} \circ F$.
\end{lemma}

\begin{proof}
Since $X_{\alpha}$ coincides with its precise representative on the image of $\Omega\setminus S$ through $\phi_{\alpha}$ and $F$ is differentiable at such points, it suffices to prove that for $\H^{N-1}$ almost every approximate jump point $x_{\alpha}$ of $X_{\alpha}$, $F(x_{\alpha})$ is an approximate jump point of $X_{\beta}$ and the one-sided  approximate limits satisfy 
$$  X_{\beta}^+ (F(x_{\alpha}))+ X_{\beta}^- (F(x_{\alpha})) = d_{x_{\alpha}}F(X_{\alpha}^+(x_{\alpha})+ X_{\alpha}^-(x_{\alpha})).$$
We shall prove this formula for approximate jump points $x_{\alpha} \in \Omega\cap U_{\alpha}\setminus \Si$ which is sufficient for our purpose.

According to Theorem \ref{decompo}, for such a $x_{\alpha}$ there exists a Borel map $\nu : J_{X_{\alpha}} \rightarrow \S^{N-1}$ such that

$$ \lim_{\rho \downarrow 0} \fint_{B_{\rho}^{\pm}(x_{\alpha},\nu(x_{\alpha}))} \| X_{\alpha}(x)-X_{\alpha}^{\pm}(x_{\alpha})\|_2 \, d\Leb{N}(x)=0.  $$

We set $x_{\beta}= F(x_{\alpha})$, $\kappa(F(x)) =(^t(d_{x}F)) ^{-1}(\nu(x))$, and $\nu'(y)= \kappa(y) / \|\kappa(y)\|_{2}$. We have to show that

$$ \lim_{\rho \downarrow 0} \fint_{B_{\rho}^{\pm}(x_{\beta}, \nu'(x_{\beta}))} \| X_{\beta}(y)-d_{x_{\alpha}}F(X_{\alpha}^{\pm}(x_{\alpha}))\|_2 \, d\Leb{N}(y)=0.  $$

To this aim, we point out that $F^{-1}(B_{\rho}(x_{\beta})) \subset B_{Lip (F^{-1})\rho}(x_{\alpha})$
 and $F^{-1}(y) \in B_{\rho}^{\pm}(x_{\alpha},\nu(x_{\alpha}))$ whenever 
 $y \in B_{\rho}^{\pm}(x_{\beta}, \nu'(x_{\beta}))$ and $\rho$ is sufficiently small. Indeed, since $x_{\beta}$ is a regular point by assumption, $F^{-1}$ is differentiable at $x_{\beta}$; hence

\begin{align*}
\langle F^{-1}(y)-F^{-1}(x_{\beta}), \nu(x_{\alpha}) \rangle &= \langle d_{x_{\beta}}F^{-1}(y-x_{\beta}), \nu(x_{\alpha})\rangle + o(\|y- x_{\beta}\|_2) \\
																			&= \langle y-x_{\beta}, ^t(d_{x_{\beta}}F^{-1}) \nu(x_{\alpha}) \rangle +o(\|y- x_{\beta}\|_2) \\
																			&=   \|\kappa(x_{\beta})\|_{2}\langle y-x_{\beta}, \nu'(x_{\beta})\rangle +o(\|y- x_{\beta}\|_2). 
\end{align*}

Therefore, applying the change of variable formula, we get

\begin{align*}
 \int_{B_{\rho}^{\pm}(x_{\beta}, \nu'(x_{\beta})) }&\| X_{\beta}(y)-d_{x_{\alpha}}F(X_{\alpha}^{\pm}(x_{\alpha}))\|_2 \, d\Leb{N}(y) \leq \\
 &\int_{B_{Lip (F^{-1})\rho}^{\pm}(x_{\alpha}, \nu(x_{\alpha}))} \| X_{\beta} \circ F(x)-d_{x_{\alpha}}F(X_{\alpha}^{\pm}(x_{\alpha}))\|_2 \,|\det d_xF| d\Leb{N}(x) \\
 &  \leq (Lip \,F)^N \int_{B_{Lip (F^{-1})\rho}^{\pm}(x_{\alpha}, \nu(x_{\alpha}))} \| d_xF(X_{\alpha}(x))-d_{x_{\alpha}}F(X_{\alpha}^{\pm}(x_{\alpha}))\|_2 \,d\Leb{N}(x) 
\end{align*}
by (\ref{standvf}). We conclude by using that $d_xF$ is continuous at $x_{\alpha}$ and $x_{\alpha}$ is an approximate jump point of $X_{\alpha}$.
\end{proof}

\subsubsection{Covariant derivative of $\GBV$ vector fields}

As for covariant tensors, we start with the definition when read in a chart.

 \begin{defi}[Covariant derivative of a $\GBV$ vector field in a chart $\hat{V}\subset \R^N$]
 Given a vector field $Y$ on $\hat{V}= \phi(V)$ with $\GBV$ components, we define a $(1,1)$-tensor $DY$ by the formula
$$ DY = \sum_{j,s} (DY)_{j}^s \, dy^j \otimes \frac{\partial}{\partial y_s},$$
where 
\begin{equation}\label{CoDeVF} (DY)_{j}^s = \frac{\partial Y_s}{\partial y_{j}}   +  \sum_{v=1}^N  Y_v \,\displaystyle \Gamma_{jv}^s.
\end{equation} 
The components of $DY$ belong to $\GM(\hat{V})$.
 \end{defi}

\begin{defi}[Pull-back of $(1,1)$-tensor]
Given a $DC_0$ transition map $F: \hat{U} \longrightarrow \hat{V}$ and a  $(1,1)$-tensor $S$ a on $\hat{V}$
$$ S= \sum_{j,s} S_{j}^s \, dy^j \otimes \frac{\partial}{\partial y_s}$$  
with components $S_j^s$ in $\GBV(\hat{V})$ (resp. $\GM(\hat{V})$), we define a $(1,1)$-tensor $F^*S$ on $\hat{U}$ by the formula
\begin{equation}
F^*S= \sum_{i,j,k,s}  \frac{\partial F_{j}}{\partial x_{i}}  \frac{\partial F_{k}^{-1}}{\partial y_{s}}\circ F \,F^*(S_j^s)\, dx^{i} \otimes \frac{\partial}{\partial x_k}.
\end{equation}
\end{defi}
We refer to \cite{SpivakI} for a general definition for $(p,k)$-tensors. As for covariant tensors, in order to infer from this local definition a well-defined notion of covariant derivative of $\GBV$ vector field on any open subset of $X^*$, it remains to prove a compatibility formula.
 
\begin{prop} Let $F: \hat{U} \rightarrow \hat{V}$ be a $DC_0$ transition map. 
Let $Y$ be a vector field with components in $\GBV(\hat{V})$. Then, the following compatibility formula holds
\begin{equation}
D(F^*Y) = F^*(DY).
\end{equation}
\end{prop}

The proof is very similar to that of Proposition~\ref{compa}, thus the details are left to the reader.

%%%%%%%%%%%%%%%%%%%%%%%%%%%%%%%%%%
%%%%%%%%%%%%%%%%%%%%%%%%%%%%%%%%%%

\subsection{Norm of tensors}

In order to define the norm of tensor with Radon measures components, we cannot apply the standard definition for smooth tensors on smooth Riemannian manifolds (see for instance \cite[Section 6.3]{petersen}). Indeed, the classical definition involves products of components of the tensor and thus does not apply to our setting. However, in the case when the tensor components are functions, our definition coincides with the standard one.

\subsubsection{Evaluation of tensors}
\begin{defi}[Evaluation of tensors] Given $\Omega\subset X^*$ open, a covariant $p$-tensor $S$ with $\GM$ components on $\Omega$, and $X_1,\cdots,X_p$  vector fields with  $\Cwo(\Omega)$ components, we may evaluate the tensor on the vector fields by the formula
in any local coordinates
$$ S_{\beta}(X_{\beta,1},\cdots, X_{\beta,p}) = \sum_{\underline{j}} X_{\beta,1}^{j_1}\cdots X_{\beta,p}^{j_p} S_{\beta,\underline{j}} .$$
\end{defi}

Note that it is also possible to evaluate tensors by $\GBV$ vector fields. However, it is necessary to use the precise representatives of the vector fields for the formula to make sense, since $S_{\beta,\underline{j}}$ need not be absolutely continuous w.r.t. $\H^N$. Whatever the regularity of the vector fields, what really matters is that 
they are defined everywhere out of a $\H^{N-1}$-negligible set and bounded in the sense of Definition~\ref{def:cw}. 

The following lemma shows that this definition satisfies a compatibility condition. 

\begin{lemma}
\label{tensoreval} Suppose we are given a covariant $p$-tensor $S$ with $\GM$ components defined on an open set $\Omega \subset X^*$ and $\GBV$or $\Cwo$ vector fields $X_1, \cdots, X_p$ on $\Omega$. Then, for any pair of charts $(\phi_{\alpha},U_{\alpha})$, $(\phi_{\beta},U_{\beta})$ such that $U_{\alpha} \cap U_{\beta} \cap \Omega \neq \emptyset$, the following compatibility conditions holds
$$ F^* (S_{\beta} (X_{\beta,1},\cdots, X_{\beta,p}))= S_{\alpha}(X_{\alpha,1},\cdots, X_{\alpha,p})
\qquad\text{in $\phi_\alpha(U_\alpha\cap U_\beta\cap\Omega)$.}$$
Therefore the collection of measures $(S_{\alpha}(X_{\alpha,1},\cdots, X_{\alpha,p}))_{\alpha \in \Lambda}$ in $\GM$ is
 a system of Radon measures on $\Omega$ according to Definition~\ref{def:sysrado}.
\end{lemma}

\begin{remark}
Note that given $f$ a $\GBV$ function and $X$ a $\GBV$ vector field defined on $\Omega$, the above result and Lemma \ref{MeasFSyst} yield that $df(X)$ is a well-defined Radon measure on $\Omega$.
\end{remark}
\begin{proof}
By definition,
$$ S_{\beta}(X_{\beta,1},\cdots, X_{\beta,p}) = \sum_{\underline{j}} X_{\beta,1}^{j_1}\cdots X_{\beta,p}^{j_p} S_{j_1\cdots j_p}.$$
In other terms, $S_{\beta}(X_{\beta,1},\cdots, X_{\beta,p})$ is a sum of measures where each summand is a product of a locally bounded Borel function and a Radon measure in $\GM(\phi_\alpha(U_\alpha\cap U_\beta\cap\Omega))$. Therefore, applying (\ref{eq:stimamax3}), we infer
$$ F^*(S_{\beta}(X_{\beta,1},\cdots, X_{\beta,p}))= \sum_{\underline{j}} (X_{\beta,1}^{j_1}\circ F)\cdots (X_{\beta,p}^{j_p}\circ F) F^*(S_{\beta,\underline{j}}).$$
Now, by definition of vector field (see (\ref{standvf})) and according to Lemma \ref{prVfield}, we have the following equality up to a $\H^{N-1}$-negligible set
$$
X_{\beta,k}^{j_k} \circ F  
	= \sum_{1\leq i_k\leq N} \frac{\partial F_{j_k}}{\partial x_{i_k}} X_{\alpha,k}^{i_k}.
$$
Besides, by definition of $F^*$ on tensors,
$$\sum_{\underline{j}} \frac{\partial F_{j_1}}{\partial x_{i_1}}\cdots \frac{\partial F_{j_p}}{\partial x_{i_p}} F^*(S_{\beta, \underline{j}}) 
= (F^*(S_{\beta}))_{i_1,\cdots,i_p}.$$
By combining all of this, we finally get
\begin{align*}
F^*(S_{\beta}(X_{\beta,1},\cdots, X_{\beta,p})) &= \sum_{\underline{i}} (F^*(S_{\beta}))_{\underline{i}} X_{\alpha,1}^{i_1} \cdots X_{\alpha,p}^{i_p}\\
								&= \sum_{\underline{i}}  (S_{\alpha})_{\underline{i}} X_{\alpha,1}^{i_1} \cdots X_{\alpha,p}^{i_p} \\
							&= S_{\alpha}(X_{\alpha,1},\cdots, X_{\alpha,p}),
\end{align*}
and the proof is complete.
\end{proof}

\begin{remark}\label{evalvf}
It is also possible to evaluate a covariant $1$-tensor with $\GBV$ components on a vector field with $\GM$ components. The above proof applies with minor changes.
\end{remark}
%%%%%%%%%%%%%%%%%%%%%%%%%%%%%%%%%%%%%%%%%%%%%%%%%%%%%%%%%%%%%%%%
%%%%%%%%%%%%%%%%%%%%%%%%%%%%%%%%%%%%%%%%%%%%%%%%%%%%%%%%%%%%%%%%

\subsubsection{Local definition of the norm of tensor}

We first define the norm of a tensor in a chart. Recall that the existence of local orthonormal frames made of $\Cwo(\Si)$ vector fields
is guaranteed by Remark~\ref{gramschmidt}.

\begin{defi}[Local definition of $|S|_g$]\label{Tensornorm}Let  $S$ be a covariant $p$-tensor defined on an open set
$ \Omega \subset X^*$ with $\GM(\Omega)$ components and $(U_{\alpha},\phi_{\alpha})$ a chart of $X^*$. 
 Let $(E_i)_{1\leq i\leq N}$ be a local orthonormal frame in $\phi_{\alpha}(\Omega\cap U_{\alpha})$ with $E_i\in\Cwo(\phi_{\alpha}(\Omega\cap U_{\alpha}), \phi_{\alpha}(\Si\cap U_{\alpha}))$. 
 We define the norm $|S_{\alpha}|_g$ of $S_{\alpha}$ as the total variation of the $\R^{N^p}$-valued Radon measure defined as
\begin{equation}\label{eq:july4}
(S (E_{i_1},\cdots, E_{i_p}))_{(i_1,\cdots, i_p) \in \{1,\cdots, N\}^p}.
\end{equation}
\end{defi}

Here and in the sequel, we always assume that local orthonormal frames are $\Cwo$ regular out of the image of $\Si$ through the chart. Next, we show that Definition~\ref{Tensornorm} is well posed.

\begin{lemma}Let  $S$ be a covariant $p$-tensor defined on $ \Omega \subset X^*$ with  $\GM(\Omega)$ components and 
let $(U_{\alpha},\phi_{\alpha})$ be a chart of $X^*$. Then, the norm $|S_{\alpha}|_g$ does not depend on the choice of the local frame.  
\end{lemma}

\begin{proof} Let $(E_i)_{1 \leq i \leq N}$, $(\tilde{E}_i)_{1 \leq i \leq N}$ be local orthonormal frames on $\Omega$
and let $P:\Omega\to\R^{N^2}$ be the change-of-basis matrix with respect to $(E_i)$ and $(\tilde{E}_i)$, i.e.
$\tilde E_i^j=\sum_\ell P_{i\ell} E^\ell_j$. We denote by $\sigma=|\lambda|$ the total variation of the vector-valued measure in \eqref{eq:july4} and by
$\tilde\sigma$ and $\tilde\lambda$ the analogous quantities for $\tilde{E}_i$. Our goal is to show that
$\tilde\sigma$ and $\sigma$
coincide as measures in $\Omega$. To this aim, we will find a representation of $\tilde\lambda$
in terms of $\lambda$. Let $T:\Omega\to\R^{N^p}$ be a Borel unit vector field providing the polar representation
of $\lambda$, i.e.
$$
S (E_{i_1},\cdots, E_{i_p})=T_{\underline i}\, \sigma\quad\qquad \mbox{ for all } \underline i=(i_1,\cdots,i_p) \in \{1,\cdots,N\}^p
$$  
and, writing $S=\sum_{\underline j}S_{\underline j} \,dy^{j_1}\otimes\cdots\otimes dy^{j_p}$, let us compute:
\begin{eqnarray*}
S (\tilde E_{i_1},\cdots, \tilde E_{i_p})&=&\sum_{\underline j}S_{\underline j}\tilde E_{i_1}^{j_1}\cdots \tilde E_{i_p}^{j_p}\\
&=&
\sum_{\ell_1,\cdots,\ell_p=1}^N\sum_{\underline j}S_{\underline j} P_{i_1\ell_1}E^{\ell_1}_{j_1}\cdots P_{i_p\ell_p}E^{\ell_p}_{j_p}\\
&=&
\sum_{\ell_1,\cdots,\ell_p=1}^NP_{i_1\ell_1}\cdots P_{i_p\ell_p} S (E_{\ell_1},\cdots, E_{\ell_p})\\
&=&
\sum_{\ell_1,\cdots,\ell_p=1}^NP_{i_1\ell_1}\cdots P_{i_p\ell_p} \,T_{\ell_1,\cdots,\ell_p}\,\sigma.
\end{eqnarray*}
Setting $(P^*)_{\underline i\underline\ell}=P_{i_1\ell_1}\cdots P_{i_p\ell_p}$, this proves that 
$$
\tilde\lambda_{\underline i}=\sum_{\underline\ell} (P^*)_{\underline i\underline\ell} T_{\underline\ell}\,\sigma.
$$
Since $P$ is an orthogonal matrix and $\|T(x)\|_2=1$, it is easily seen that $\|P^* T(x)\|_2=1$, therefore $|\tilde\lambda|=\sigma$.
\end{proof}

We can now check that $|S|_g$ is a well-defined Radon measure on $\Omega$.

\begin{lemma}\label{measuretensornorm} Let  $S$ be a covariant $p$-tensor defined on $ \Omega \subset X^*$ with  $\GM(\Omega)$ components. Then $(|S_{\alpha}|_g)_{\alpha \in \Lambda}$ is a system of $\GM$ measures. 
Thus, it induces a measure $|S|_g \in \GM(\Omega)$ defined by (\ref{eq:eccocosae2}) called the norm of $S$.
\end{lemma}

\begin{proof}
By definition of tensor, for any pair of charts $(\phi_{\alpha},U_{\alpha})$, $(\phi_{\beta},U_{\beta})$ such that $U_{\alpha} \cap U_{\beta} \cap \Omega \neq \emptyset$, the following compatibility conditions holds
$$ F^* ( S_{\beta} (E_{\beta,1},\cdots, E_{\beta,p}))= S_{\alpha}(E_{\alpha,1},\cdots, E_{\alpha,p})$$
where $(E_{\alpha,1},\cdots, E_{\alpha,N})$ and $(E_{\beta,1},\cdots, E_{\beta,N})$ are local orthormal frames in $\phi_{\alpha}(U_{\alpha}\cap \Omega)$ and $\phi_{\beta}(U_{\beta} \cap \Omega)$ respectively. The result then follows from Lemmas \ref{starpush} and \ref{MeasFSyst}.
%Thus, we only have to check that
%$$ F^* ( |S_{\beta} (E_{\beta,1},\cdots, E_{\beta,p})|_g)= |S_{\alpha}(E_{\alpha,1},\cdots, E_{\alpha,p})|_g.$$
\end{proof}

\begin{remark} Let $f \in \GBV(\Omega)$, being $\Omega \subset X^*$ an open set. Then, the differential of $f$ is a covariant 1-tensor with $\GM(\Omega)$ components. As a consequence, the above result applies and gives us an intrinsic notion of  ``total variation'' $|df|_g$ of the measure-valued 1-form $df$. 
\end{remark}

The arguments above can easily be adapted to show that there is a well-defined notion of norm for vector field with $\GM$ components. We give the formal definition below and leave the details to the reader.

\begin{defi}[Norm of a vector field with $\GM$ components]\label{NormMeasVF}
Let $X$ be a vector field with $\GM(\Omega)$ components, being $\Omega \subset X^*$ an open set. Then the norm $|X|_g \in \GM(\Omega)$ of $X$ is defined as the nonnegative measure induced by the following system of measures. These local measures are defined in each chart whose domain intersects $\Omega$ as the total variation of the $\R^N$-valued Radon measure $(\theta^1,\cdots,\theta^N)$ where $(\theta^i)_{1\leq i\leq N}$ are the local coordinates of $X$ with respect to a local orthonormal frame. This definition does not depend on the choice of the orthonormal frame.
\end{defi}

As in the smooth setting, the norm of tensor can be used in the following estimate. 

\begin{lemma}\label{CSTensor}
Suppose we are given a covariant p-tensor $S$ with $\GM$ components defined on an open set $\Omega \subset X^*$ and $\GBV$ or $\Cwo$ vector fields $X_1, \cdots, X_p$ on $\Omega$. Then, the following inequality between measures holds
$$ |S(X_1,\cdots, X_p) | \leq   |X_1|_g \cdots |X_p|_g |S|_g,$$
where $|X_i|_g= \sqrt{g(X_i,X_i)}$.
\end{lemma}
\begin{proof}
Since we proved that $|S(X_1,\cdots, X_p) | $ and $ |S|_g$ are both systems of Radon measures, we can work in local coordinates $(U,\phi)$ omitting the $\cdot_{\alpha}$ for simplicity. Let $(E_i)$ be a local orthonormal frame. Then, according to Definition \ref{Tensornorm}, the Radon-Nikodym theorem implies the existence of a $\R^{N^p}$-valued map $M(x)$ with unit Euclidean norm such that 
$$ (S(E_{i_1},\cdots, E_{i_p}))_{i_1 \cdots, i_p} = M(x)   |S|_g. $$
Decomposing $X_1,\cdots, X_p$ in the basis $(E_i)$ as $ X _s=\sum_{i=1}^N \theta_s^i E_i $,
yields for any Borel set $A$
\begin{align*}
 |S(X_1,\cdots, X_p) (A)| & = |\sum_{i_1 \cdots, i_p}  \theta_1^{i_1} \cdots\theta_p^{i_p}  S(E_{i_1},\cdots,E_{i_p})(A)| \\
 							& \leq \Big|\int_A \sum_{i_1 \cdots, i_p}  \theta_1^{i_1}(x) \cdots\theta_p^{i_p}(x)  M_{i_1 \cdots, i_p}(x) \,d|S|_g (x)\Big|\\
 							& \leq \int_A  \sum_{i_1 \cdots, i_p}  |\theta_1^{i_1}(x)| \cdots|\theta_p^{i_p}(x)| \, |M_{i_1 \cdots, i_p}|(x) \,d|S|_g (x)\\
 							& \leq \int_A  |X_1|_g\cdots| X_p|_g\,d|S|_g (x)\\
 							& \leq  |X_1|_g\cdots| X_p|_g |S|_g(A),
 \end{align*}
where to get the penultimate inequality we use the Cauchy-Schwarz inequality, the fact that $M(x)$ has unit Euclidean norm, and that $(E_i)$ is a local orthormal frame for $g$, hence $ \sum_{i=1}^N (\theta_s^i)^2= |X_s|_g^2$. Consequently, since $A$ is arbitrary, we get the result by the definition of total variation.  
\end{proof}

%%%%%%%%%%%%%%%%%%%%%%%%%%%%%%%%%%
%%%%%%%%%%%%%%%%%%%%%%%%%%%%%%%%%%
%%%%%%%%%%%%%%%%%%%%%%%%%%%%%%%%%%
%%%%%%%%%%%%%%%%%%%%%%%%%%%%%%%%%%

\section{Hessian and Laplacian of a $DC$ function}\label{laplacian}
In this section, we use the tensor calculus we introduced in order to define the Hessian, its norm, and the Laplacian of a $DC$ function. Let $f: \Omega \subset X^* \longrightarrow \R$ (with $\Omega$ an open set) be a $DC$ function. Then $df$ is a covariant 1-tensor with $\GBV$ components whose representation in a chart $(U,\phi)$ is
$$ d(f \circ \phi^{-1})= \sum_{i=1}^n \frac{\partial (f \circ \phi^{-1})}{\partial x_i} dx^i.$$

\subsection{Hessian of a $DC$ function}

We can define the Hessian of $f$ as follows.

\begin{defi}[Hessian of a $DC$ function]\label{hessian} Let $X^* $ be a $DC_0$ Riemannian manifold and let $f$ be a $DC$ function 
defined on an open subset of $X^*$. We define the Hessian of $f$ as the covariant 2-tensor
$$ \hes f = D df.$$

When read in a chart $(U,\phi)$, the components of $\hes f$ are given by
$$  (\hes_\phi f)_{ij}: = 
\frac{\partial^2 (f\circ \phi^{-1}) }{\partial x_i \partial x_j} - \sum_{k=1}^N \frac{\partial (f \circ \phi^{-1})}{\partial x_k} \Gamma	_{ij}^k.$$
The norm of $\hes f$ is then defined as in Lemma \ref{measuretensornorm}.
\end{defi} 

\begin{remark} Note that the components $(\hes_\phi f)_{ij}$ of $\hes f$ in local coordinates
are symmetric with respect to $i,\,j$. Therefore, the Hessian of a $DC$ function can also be considered as a symmetric matrix-valued Radon measure.
\end{remark}

\begin{prop}\label{FormHessian}
Let $\Omega \subset X^*$ be an open set, $f \in DC(\Omega)$, and $X,\,Y \in \GBV(\Omega)$. Then, the following equality in $\GM(\Omega)$ holds
\begin{equation}\label{eq:ihp2}
\hes f (X,Y)= D(df(Y))(X)- df(D_XY).
\end{equation}
\end{prop}

\begin{remark} $df$ is a covariant 1-tensor with $\GBV$ components, therefore it makes sense to evaluate it on a $\GM$ vector field, 
see Remark~\ref{evalvf}. $D(df(Y))$ stands for the derivative of the $\GBV$ function $df(Y)$ (which is usually {\it not} written $d(df(Y))$ to avoid confusion with the exterior derivative of differential forms).
\end{remark}
\begin{proof}
It is sufficient to check that both measures coincide when read in a chart. For simplicity, we keep the same notations for the function and the vector fields read in a chart.
Let us recall that read in a chart,
\begin{equation}\label{Form1}
\hes f(X,Y)= \sum_{1\leq i,j\leq N} X_iY_j \Big( \frac{\partial^2 f }{\partial x_i \partial x_j} - \sum_{k=1}^N \frac{\partial f }{\partial x_k} \Gamma	_{ij}^k\Big),
\end{equation}
while the right-hand side of \eqref{eq:ihp2} is the sum of
\begin{equation}\label{Form2}
D(df(Y))(X)= \sum_{1\leq i,j\leq N} X_i\frac{\partial }{\partial x_i}\Big(\frac{\partial f}{\partial x_j}Y_j\Big)
\end{equation}
and, using $D_XY= \sum_{1\leq i,j\leq N} X_i(\frac{\partial Y_j}{\partial x_{i}}   +  \sum_{s=1}^N  Y_s \,\Gamma_{is}^j )\frac{\partial}{\partial x_j}$
\begin{equation}\label{Form3}
-df(D_XY) = -\sum_{1\leq i,j\leq N} \frac{\partial f}{\partial x_j}X_i\frac{\partial Y_j}{\partial x_i} - \sum_{1\leq i,j,s\leq N} X_iY_s{\frac{\partial f}{\partial x_j}}\Gamma_{is}^j.
\end{equation}

From (\ref{Form1}), (\ref{Form2}) and (\ref{Form3}), it is then clear that the absolutely continuous part and the Cantor part of both measures coincide since
$$ D^i(uv)= uD^iv +vD^iu$$
when $u,\,v$ are $BV$ functions and $D^i$ stands for  either the absolutely continuous part or the Cantor part of the derivative, we refer to \cite{AFP} for a proof. Therefore, it remains to study the jump part of the measures. Let us recall that, according to Theorem~\ref{decompo}(d), 
the jump part of the derivative of $f\in BV$ has the following structure
$$ D^{ju} f = (f^+-f^-)\nu_f \,d\H^{N-1}\res{J_f}$$
where $\nu_f: J_f\rightarrow \S^{N-1}$ is a Borel function, the so-called approximate unit normal, and $J_f$ is the set of approximate jump points of $f$. 

To this aim, we fix $i,\,j \in \{1,\cdots,N\}$ and compute the jump part of $\frac{\partial }{\partial x_i}\big(\frac{\partial f}{\partial x_j}Y_j\big)$ appearing in (\ref{Form2}). According to the chain rule formula, we have to consider three cases. Below, we write $\cdot^{ju}$ for the jump part of the partial derivative and $\langle \cdot,\cdot\rangle$ for the standard Euclidean inner product. We also use many times that we are considering the precise representatives of $Y_j$ and $\frac{\partial f}{\partial x_j}$, which are ``undefined'' only on the $\H^{N-1}$-negligible set of approximate discontinuity points which are not approximate jump points. 

First, on $J_{Y_j}\setminus J_{\partial f /\partial x_j}$ one has 
$$\frac{\partial^{ju}}{\partial x_i} \big(\frac{\partial f}{\partial x_j}Y_j\big) = 
\frac{\partial f}{\partial x_j} (Y_j^+-Y_j^-) \,\langle \partial/\partial x_i,\nu_{Y_j}\rangle \,d\H^{N-1}\res{(J_{Y_j}\setminus J_{\partial f /\partial x_j}}).$$
Second, on $J_{\partial f /\partial x_j}\setminus J_{Y_j}$ one has 
$$\frac{\partial^{ju}}{\partial x_i}\big(\frac{\partial f}{\partial x_j}Y_j\big) = 
Y_j\big( \frac{\partial f}{\partial x_j}^+ -  \frac{\partial f}{\partial x_j}^-\big)
 \,\langle \partial/\partial x_i,\nu_{\partial f /\partial x_j}\rangle \,d\H^{N-1}\res{(J_{\partial f /\partial x_j}\setminus J_{Y_j}}).$$ 
Last, assuming with no loss of generality that $\nu_{Y_j}=\nu_{\partial f /\partial x_j}$ $\H^{N-1}$-a.e. on $J_{\partial f /\partial x_j}\cap J_{Y_j}$ (here we use
the fact that the approximate unit normals coincide up to the sign $\H^{N-1}$-a.e. on the intersection of approximate
jump points, see \cite[Example 3.97]{AFP}, and so we can assume that they coincide, up to a permutation of the right and left approximate limits),
on $J_{\partial f /\partial x_j}\cap J_{Y_j}$ one has
$$\frac{\partial^{ju}}{\partial x_i}\big(\frac{\partial f}{\partial x_j}Y_j\big) = 
\bigl(Y_j^+\frac{\partial f}{\partial x_j}^+  -  Y_j^-\frac{\partial f}{\partial x_j}^-\bigr)
 \, \langle \partial/\partial x_i,\nu_{\partial f /\partial x_j} \rangle\,d\H^{N-1}\res{(J_{\partial f /\partial x_j}\cap J_{Y_j})}.$$ 

According to (\ref{Form2}), (\ref{Form3}), we have to compute the density $\theta$ with respect to $\H^{N-1}$ of
$$\frac{\partial^{ju}}{\partial x_i}\big(\frac{\partial f}{\partial x_j}Y_j\big) - \frac{\partial f}{\partial x_j}\frac{\partial^{ju}}{\partial x_i} Y_j.$$
On $J_{Y_j}\setminus J_{\partial f /\partial x_j}$ we have
$$\theta   =\Big(\frac{\partial f}{\partial x_j} (Y_j^+-Y_j^-) -  \frac{\partial f}{\partial x_j} (Y_j^+-Y_j^-)\Big)
\,\langle \partial/\partial x_i,\nu_{Y_j}\rangle=
0.$$
On $J_{\partial f /\partial x_j}\setminus J_{Y_j}$ we have
$$
\theta
= \Big(\big(\frac{\partial f}{\partial x_j}^+-\frac{\partial f}{\partial x_j}^-\big)Y_j\Big)
\,\langle \partial/\partial x_i,\nu_{\partial f /\partial x_j}\rangle.
$$
Finally, on $J_{\partial f /\partial x_j}\cap J_{Y_j}$, writing $\theta=\tilde\theta\langle \partial/\partial x_i,\nu_{\partial f /\partial x_j}\rangle=
\tilde\theta\langle \partial/\partial x_i,\nu_{Y_j}\rangle$, on this intersection we have
\begin{align*}
\tilde\theta =\Big(\big(\frac{\partial f}{\partial x_j}^+Y_j^+& -\frac{\partial f}{\partial x_j}^-Y_j^-\big) - \frac{1}{2}\big(Y_j^+ -Y_j^-\big)\big(\frac{\partial f}{\partial x_j}^++\frac{\partial f}{\partial x_j}^-\big)\Big)\\
& =\frac{1}{2}\Big(\frac{\partial f}{\partial x_j}^+\big(Y_j^++Y_j^-\big) -\frac{\partial f}{\partial x_j}^-\big(Y_j^++Y_j^-)\Big) \\
& =\Big(\big(\frac{\partial f}{\partial x_j}^+ -\frac{\partial f}{\partial x_j}^-\big)Y_j\Big).
\end{align*}
Therefore in each case the measure
$$
\frac{\partial^{ju}}{\partial x_i}\big(\frac{\partial f}{\partial x_j}Y_j\big) - \frac{\partial f}{\partial x_j}\frac{\partial^{ju}}{\partial x_i} Y_j
$$ 
coincides with the jump part of  $Y_j\frac{\partial^2 f }{\partial x_i \partial x_j}$. Since $i,\,j$ are arbitrary, the proof is complete.
\end{proof}

%%%%%%%%%%%%%%%%%%%%%%%%%%%%%%%%%%
%%%%%%%%%%%%%%%%%%%%%%%%%%%%%%%%%%

\subsection{Laplacian of a $DC$ function}

In order to define $\Delta^g f$, we proceed as for tensors by defining the Laplacian locally in a chart first and then verifying the required compatibility formula.

\begin{defi}[Local definition of $\Delta^g f$]Let  $f: \Omega \subset X^* \rightarrow \R$ be a $DC$ function with $\Omega$ an open set.
We define the Laplacian $\Delta_\phi f$ of $f$ read in the chart $\phi$ as the trace of 
$\hes_\phi f$ with respect to $g$. More precisely,
if $(U,\phi)$ is a chart of $X^*$ such that $U\cap \Omega\neq \emptyset$, in the system of coordinates induced by $\phi$ one defines
$$\Delta_\phi f: =
\sum_{1 \leq i,j\leq N} g^{ij}\left(\frac{\partial^2(f\circ\phi^{-1})}{\partial x_i \partial x_j} 
- \sum_{k=1}^N \frac{\partial (f \circ \phi^{-1})}{\partial x_k} \Gamma_{ij}^k\right)\qquad\text{in $\phi(U\cap\Omega)$}.$$ 
\end{defi}

Now, we prove that the local definition $\Delta_\phi f$ of $\Delta^g f$ provided above induces a system of Radon measures.

\begin{lemma}\label{compameas}Let $F : \hat{U} \longrightarrow \hat{V}$ be a $DC_0$ transition map relative to the charts $(U,\phi)$ and $(V, \tilde{\phi})$. Let $f : \Omega \longrightarrow \R$ be a $DC$ function and suppose $\Omega \cap U \cap V\neq \emptyset$. Then,
\begin{equation}\label{eq2}
F^*(\Delta_{\tilde\phi} f )  = \Delta_\phi f .
\end{equation}
\end{lemma} 
\begin{proof} Note that in the following argument, the Hessian is considered as symmetric matrix-valued Radon measure. By definition,
 $$ F^*(\Delta_{\tilde\phi} f)  = F^*(tr (\tilde{G}^{-1}\hes_{\tilde\phi} f )).$$
We set $F^*((\hes_{\tilde\phi} f)_{\cdot \cdot})$ the symmetric matrix-valued Radon measure whose entries are 
$F^*((\hes_{\tilde\phi}f)_{ij})$ in the sense of Definition~\ref{pbmf}. With this notation, we can rephrase the above equality in the following way
\begin{equation}\label{eq3}
 F^*(\Delta_{\tilde\phi} f) = tr \left(\tilde{G}^{-1} \circ F F^*((\hes_{\tilde\phi} f)_{\cdot \cdot})\right).
 \end{equation}
From the tensor equality $\hes_\phi f= F^*(\hes_{\tilde\phi} f)$, we infer the following equality of matrices
$$ \hes_\phi f = ^t (dF) F^*((\hes_{\tilde\phi} f )_{\cdot \cdot}) (dF),$$
which is equivalent to
$$F^*((\hes_{\tilde\phi} f )_{\cdot \cdot}) = {}^t(dF)^{-1} \hes_\phi f   (dF)^{-1}.$$
By combining this together with
$$\tilde{G}^{-1} \circ F  =   \left( d F\right) \big( {G}^{-1}\big) ^t\left(d F\right)$$
and then inserting these equalities in (\ref{eq3}), we obtain
$$ F^*(\Delta_{\tilde\phi} f) = tr \left(\big( {G}^{-1}\big)\hes_\phi f  \right)$$
and the proof is complete.
\end{proof}

Now, using Lemma~\ref{compameas} and Lemma~\ref{MeasFSyst}, we can define the Laplacian of a $DC$ function on $\Omega$.

\begin{defi}[Laplacian]Let $f :\Omega \subset X^* \rightarrow \R$ be a $DC$ function. We set $\Delta^g f $ the Radon measure defined by (\ref{eq:eccocosae}).
\end{defi}

Our goal is now to prove that $\Delta^g f$ coincides with the weak Laplacian defined through integration by parts. To this aim we generalize to our setting the following classical formula.

\begin{prop}\label{vollaplacian}Let $X^* $ be a $DC_0$ Riemannian manifold and $f$ be a $DC$ function defined on an open subset of $X^*$. 
Then the local expression $\Delta_\phi f$ of $\Delta^g f$ in a chart $(U,\phi)$ is given by
\begin{equation}\label{dclap2} \Delta_\phi f = \frac{1}{\sqrt{{\rm det\,}G}}\sum_{1 \leq i,j\leq N} \frac{\partial}{ \partial x_i} 
\bigg( g^{ij} \sqrt{{\rm det\,}G} \,\frac{\partial (f \circ \phi^{-1})}{\partial x_j}\bigg).
\end{equation}
\end{prop}

\begin{proof}
Let us set $\tilde f=f \circ \phi^{-1}$ for notational simplicity. We expand
\begin{align*}
\sum_{1\leq i,j,k\leq N} g^{ij}\frac{\partial \tilde f}{\partial x_k} \Gamma_{ij}^k  =& 
 							\sum_{1\leq i,j,k,l\leq N} \frac{1}{2} \frac{\partial\tilde f}{\partial x_k} g^{ij}g^{kl}\left( \frac{\partial g_{jl}}{\partial x_i}  + \frac{\partial g_{il}}{\partial x_j} - \frac{\partial g_{ij}}{\partial x_l} \right)\\
 							=& 	\sum_{1\leq i,j,k,l \leq N}		 \frac{\partial\tilde f}{\partial x_k} g^{ij}g^{kl}\left[\left( \frac{\partial g_{jl}}{\partial x_i}\right) - \frac{1}{2}  \frac{\partial g_{ij}}{\partial x_l} \right],
\end{align*}		
since $g^{ij}=g^{ji}$. By differentiating with respect to $x_i$ the equality $\sum_{1\leq j\leq N} g^{ij}g_{jl}= \delta_{il}$ in $\GBV_0$, we infer  
  $$ 			\sum_{1\leq i,j,k,l \leq N}\frac{\partial\tilde f}{\partial x_k} g^{ij}g^{kl}\left( \frac{\partial g_{jl}}{\partial x_i}\right) = - \sum_{1\leq i,j,k,l\leq N}	 
   \frac{\partial\tilde f}{\partial x_k} g_{jl}g^{kl} \left( \frac{\partial g^{ij}}{\partial x_i}\right) = - \sum_{1\leq i,j\leq N} \frac{\partial\tilde f}{\partial x_j} \frac{\partial g^{ij}}{\partial x_i},$$
so that the definition of $\Delta_\phi f$ gives
\begin{equation}\label{eq:intermediate_step}
\Delta_\phi f=
\sum_{1 \leq i,j\leq N} g^{ij}\frac{\partial^2\tilde f}{\partial x_i \partial x_j} 
+\sum_{1\leq i,j\leq N} \frac{\partial\tilde f}{\partial x_j} \frac{\partial g^{ij}}{\partial x_i}
+
\frac 12 \sum_{1\leq i,j,k,l \leq N}		 \frac{\partial\tilde f}{\partial x_k} g^{ij}g^{kl}\frac{\partial g_{ij}}{\partial x_l}.
\end{equation}

Now, we expand the right-hand term in (\ref{dclap2}) using Lemma~\ref{lemmaproduct}, and then apply (\ref{diffdet}) 
to get
\begin{align*}
\frac{1}{\sqrt{\det {G}}}\sum_{1 \leq i,j\leq N} \frac{\partial}{ \partial x_i} \bigg( g^{ij} \sqrt{{\rm det\,}G} \,\frac{\partial\tilde f }{\partial x_j}\bigg) = &
 \sum_{1 \leq i,j\leq N} g^{ij}\frac{\partial^2\tilde f}{\partial x_i \partial x_j} + \sum_{1 \leq i,j\leq N}\frac{\partial\tilde f}{\partial x_j}\frac{\partial g^{ij}}{\partial x_i} \\
 & + \sum_{1 \leq i,j\leq N}g^{ij}\frac{\partial\tilde f}{\partial x_j}\frac{1}{\sqrt{{\rm det\,}G}}\frac{\partial \sqrt{{\rm det\,}G}}{\partial x_i}\\
 = &  \sum_{1 \leq i,j\leq N} g^{ij}\frac{\partial^2\tilde f}{\partial x_i \partial x_j} + \sum_{1 \leq i,j\leq N}\frac{\partial\tilde f}{\partial x_j}\frac{\partial g^{ij}}{\partial x_i} \\
 & +  \frac{1}{2} \sum_{1\leq i,j,k,s\leq N} \frac{\partial\tilde f}{\partial x_j} g^{ij}g^{ks} \frac{\partial g_{ks}}{\partial x_i}
 \end{align*}
and the proof can be completed comparing with \eqref{eq:intermediate_step}.
\end{proof}

\begin{prop}\label{IBP}
Let $f : \Omega \subset X^* \longrightarrow \R$ be a $DC$ function and $\psi \in Lip_c (\Omega)$. Then,
$$ \int_{\Omega} g(\nabla f, \nabla \psi) \,dv_g = - \int_{\Omega} \psi \Delta^g f. $$
\end{prop}
\begin{proof}
By using a Lipschitz and locally finite partition of unity $ \sum_{\alpha \in \Lambda} \theta_{\alpha}=1$ subordinate to the atlas of $X^*$, we have
\begin{align*} 
\int \psi \,\Delta^g f = & \sum_{\alpha \in \Lambda} \int_{\phi_{\alpha }(\Omega \cap U_{\alpha})} \psi \circ \phi_{\alpha }^{-1} \theta_{\alpha} \circ \phi_{\alpha }^{-1} \, \sqrt{\det G} \, \Delta_\phi f  \\
		= & \sum_{\alpha \in \Lambda} \int_{\phi_{\alpha }(\Omega \cap U_{\alpha})} \psi \circ \phi_{\alpha }^{-1} \theta_{\alpha} \circ \phi_{\alpha }^{-1} \,\sum_{i,j} \frac{\partial}{\partial x_i} \left(g^{ij} \sqrt{{\rm det\,}G} \,\frac{\partial (f \circ \phi^{-1})}{\partial x_j} \right) \\
		= & -  \sum_{\alpha \in \Lambda}\int_{\phi_{\alpha }(\Omega \cap U_{\alpha})} \sum_{i,j} \frac{\partial (\psi\circ \phi_\alpha^{-1}\theta_{\alpha} \circ \phi_{\alpha }^{-1} ) }{\partial x_i} g^{ij} \sqrt{{\rm det\,}G} \,\frac{\partial (f \circ \phi^{-1})}{\partial x_j} \,dx\\
		= & - \sum_{\alpha \in \Lambda} \int_{\Omega}  \theta_{\alpha} \,g( \nabla \psi, \nabla f ) \,dv_g +0 
		\end{align*}
where, to get the last equality, we use the fact that the partition is locallly finite and thus for $\H^N$-a.e. $x \in \Omega$, $\nabla (\sum_{\alpha \in \Lambda} \theta_{\alpha})(x)=0$.		
\end{proof}

%%%%%%%%%%%%%%%%%%%%%%%%%%%%%%%%%%%%%%%%%%%%%%%%%%%%%%
%%%%%%%%%%%%%%%%%%%%%%%%%%%v%%%%%%%%%%%%%%%%%%%%%%%%%%%
%%%%%%%%%%%%%%%%%%%%%%%%%%%%%%%%%%%%%%%%%%%%%%%%%%%%%%
\section{Integration by parts formula for the Hessian}

In this section, we prove that the Hessian of a $DC$ function satisfies the same integration by part formula as in the $\Gamma_2$ calculus, 
see \cite{Bakry-94} for a precise definition. 

To this aim, we generalize to our setting the following classical formula. 

\begin{prop}\label{geod} Let $\psi $ be a $DC$ function. The following equality of Radon measures holds
$$ D_{\nabla \psi} \nabla \psi = \frac 12 \nabla (|\nabla \psi|_g^2) $$
 \end{prop}

\begin{proof} 
It suffices to prove the result when the objects are read in a chart $(U,\phi)$. However for simplicity, we drop $\circ\, \phi^{-1}$ and simply write $f$ and $\psi$ in the computations below.
Let us recall that in a chart, the covariant derivative of a vector field $Y$ with $\GBV$ components is defined as
$$ DY = \sum_{i,k} (DY)_{i,k} \, dx^i \otimes \frac{\partial}{\partial x_k}$$
where 
\begin{equation} (DY)_{i,k} = \frac{\partial Y_k}{\partial x_{i}}   +  \sum_{s=1}^N  Y_s \,\displaystyle \Gamma_{is}^k
\end{equation} 
Decomposing $\nabla \psi$ in the standard basis, we get
$$ \nabla \psi = \sum_{1\leq k,s\leq N} g^{ks} \frac{\partial \psi}{\partial x_s} \frac{\partial }{\partial x_k}.$$
This leads to
\begin{eqnarray*}
D_{\nabla \psi} \nabla \psi &=& \sum_{1\leq k,s\leq N} g^{ks} \frac{\partial \psi}{\partial x_s} D_{\frac{\partial}{\partial x_k}} \nabla \psi \\
&=& \sum_{1\leq k,s,l \leq N} g^{ks} \frac{\partial \psi}{\partial x_s} \left[  \frac{\partial }{\partial x_k} \left( \sum_{1\leq t\leq N} g^{lt}\frac{\partial \psi}{\partial x_t} \right)\frac{\partial }{\partial x_l}\right. \\
& & \;+ \left. \sum_{1\leq t \leq N} g^{lt} \frac{\partial \psi}{\partial x_t} D_{\frac{\partial}{\partial x_k}} {\textstyle \frac{\partial}{ \partial x_l}}\right] 
\end{eqnarray*}

Now, if $f,\,g \in \GBV(\Omega)$ are such that at least one of the two functions $f,g$ belongs to $\GBV_0(\Omega)$, then for any $i \in \{1,\cdots,N\}$, it holds
$$ \frac{\partial fg}{\partial x_i} = f \frac{\partial g}{\partial x_i} + g \frac{\partial f }{\partial x_i}.$$ 

This yields
 \begin{equation}\label{eqIPPHessian2} 
D_{\nabla \psi} \nabla \psi = A+ B,
 \end{equation}
 where
 $$ A=  \sum_{1\leq k,s,l,t\leq N} g^{ks}g^{lt}  \frac{\partial \psi}{\partial x_s}\frac{\partial^2 \psi}{\partial x_k\partial x_t} \frac{\partial}{\partial x_l}$$
 $$ B =  \sum_{1\leq k,s,l,t\leq N} g^{ks}\frac{\partial \psi}{\partial x_s}\frac{\partial \psi}{\partial x_t} \left( \frac{\partial g^{lt}}{\partial x_k} \frac{\partial }{\partial x_l}+ g^{lt} D_{\frac{\partial}{\partial x_k}} {\textstyle \frac{\partial}{ \partial x_l}}\right). $$
 
Now, we compute $  g(D_{\nabla \psi} \nabla \psi,\frac{\partial}{\partial x_i})$ for a fixed $i \in \{1,\cdots,N\}$. We notice that
$$ g(B,{\textstyle\frac{\partial}{\partial x_i}})=  \sum_{1\leq k,s,l,t\leq N} g^{ks}\frac{\partial \psi}{\partial x_s}\frac{\partial \psi}{\partial x_t}  C_{l,t,i}$$
where 
$$ C_{l,t,i}= \frac{\partial g^{lt}}{\partial x_k} g_{li} + g^{lt} \sum_{1\leq u\leq N} \Gamma_{kl}^u \,g_{ui}.$$
Using that 
\begin{equation}\label{BVtrick}
\sum_{1\leq l \leq N} g_{li}g^{lt}= \delta_{it}
\end{equation}
hence a constant function (here $\delta$ stands for the Kronecker symbol) whose derivative is null, we can rewrite
$$ \sum_{1\leq l\leq N} C_{l,t,i}= \sum_{1\leq l\leq N} \Big(-\frac{\partial g_{li}}{\partial x_k} g^{lt}+ g^{lt} \sum_{1\leq u\leq N} \Gamma_{kl}^u \,g_{ui}\Big).$$
Now, since $ \Gamma_{kl}^u= \frac 12 \Big( \sum_{1\leq \theta\leq N} g^{\theta u} \big(\frac{\partial g_{l\theta}}{\partial x_k} +\frac{\partial g_{k\theta}}{\partial x_l} -\frac{\partial g_{kl}}{\partial x_{\theta}}\big)\Big) $, we get by expanding $\sum_{1\leq u\leq N} \Gamma_{kl}^u \,g_{ui}$ 
$$ \sum_{1\leq l\leq N} C_{l,t,i}= \sum_{1\leq l\leq N}  
\frac 12 \,g^{lt}\Big(\frac{\partial g_{ki}}{\partial x_l} - \frac{\partial g_{li}}{\partial x_k} - \frac{\partial g_{kl}}{\partial x_i}\Big).$$

We finally obtain
\begin{eqnarray*}
g(B,{\textstyle\frac{\partial}{\partial x_i}})  &=&   
\sum_{1\leq k,s,l,t\leq N} \frac 12 \,g^{ks} g^{lt}\frac{\partial \psi}{\partial x_s}\frac{\partial \psi}{\partial x_t} \Big(\frac{\partial g_{ki}}{\partial x_l} - \frac{\partial g_{li}}{\partial x_k} - \frac{\partial g_{kl}}{\partial x_i}\Big) \\
&=&  \sum_{1\leq k,s,l,t\leq N} - \frac 12 \,g^{ks} g^{lt}\frac{\partial \psi}{\partial x_s}\frac{\partial \psi}{\partial x_t} \frac{\partial g_{kl}}{\partial x_i}
\end{eqnarray*}
since the other terms cancel because of the symmetries. Using (\ref{BVtrick}) again yields
\begin{equation}\label{eqIPPHessian3}
g(B,{\textstyle\frac{\partial}{\partial x_i}})  = \sum_{1\leq t,s\leq N} \frac 12 \,\frac{\partial \psi}{\partial x_s}\frac{\partial \psi}{\partial x_t} \frac{\partial g^{ts}}{\partial x_i}.
\end{equation}

The next step is to compute $\frac{1}{2}\, \frac{\partial |\nabla \psi|_g^2}{\partial x_i}$. Starting from the equality 
$$|\nabla \psi |_g^2= \sum_{1\leq t,s\leq N}  g^{ts}\frac{\partial \psi}{\partial x_s}\frac{\partial \psi}{\partial x_t}, $$
 we infer
\begin{equation}\label{eqIPPHessian1}
 {\textstyle \frac{1}{2}\, \frac{\partial |\nabla \psi|_g^2}{\partial x_i} = g(B,{\textstyle\frac{\partial}{\partial x_i}})  + \frac{1}{2} \sum_{1\leq t,s\leq N}  g^{ts} \frac{\partial}{\partial x_i}\Big(\frac{\partial \psi}{\partial x_s}\frac{\partial \psi}{\partial x_t}\Big) }.
 \end{equation}
 
 By comparing (\ref{eqIPPHessian1}) with (\ref{eqIPPHessian2}) and (\ref{eqIPPHessian3}), we get 
 \begin{equation}\label{eqIPPHessian4}
 g(D_{\nabla \psi} \nabla \psi,{\textstyle\frac{\partial}{\partial x_i}}) -  {\textstyle \frac{1}{2}\, \frac{\partial |\nabla \psi|_g^2}{\partial x_i}= 
\sum_{1\leq k,s,\leq N} g^{ks} \frac{\partial \psi}{\partial x_s}\frac{\partial^2 \psi}{\partial x_k\partial x_i} -\frac{1}{2} \sum_{1\leq k,s\leq N}  g^{ks} \frac{\partial}{\partial x_i}\Big(\frac{\partial \psi}{\partial x_s}\frac{\partial \psi}{\partial x_k}\Big) }
\end{equation} 
 where we use (\ref{BVtrick}) to simplify the first term on the right-hand side. Now, if $\psi$ were a $DC_0$ function then it would be no jump part in its second derivatives and we would get 
 $$  \frac{\partial}{\partial x_i}\Big(\frac{\partial \psi}{\partial x_s}\frac{\partial \psi}{\partial x_k}\Big)=   \frac{\partial^2 \psi}{\partial x_i\partial x_s}\frac{\partial \psi}{\partial x_k} + \frac{\partial^2 \psi}{\partial x_i\partial x_k}\frac{\partial \psi}{\partial x_s}. 
$$ 
  Since the second distributional derivative of a function is a symmetric matrix-valued Radon measure, the result is proved in this special case. But $\psi$ is not $DC_0$ in general, thus we have to compute the jump part of the term in the right-hand side of (\ref{eqIPPHessian4}) and prove that it vanishes. By symmetry of the second derivative, for any $t,s\in \{1,\cdots,N\}$, the jump part of $\frac{\partial^2 \psi}{\partial x_s\partial x_t}$ and $\frac{\partial^2 \psi}{\partial x_t\partial x_s}$ coincide. Let us also recall properties of the jump part ; in the following $\nabla^E \psi$ stands for the Euclidean gradient of $\psi$. First, for $\H^{N-1}$-a.e. point in $J_{\frac{\partial \psi}{\partial x_k}} \cap J_{\frac{\partial \psi}{\partial x_s}}$, $\nu_{\frac{\partial \psi}{\partial x_k}}= \nu_{\frac{\partial \psi}{\partial x_s}}= \nu_{\nabla^E \psi}$ and up to a $\H^{N-1}$ negligible set $J_{\nabla^E \psi}= \cup_{u=1}^N  J_{\frac{\partial \psi}{\partial x_u}}$. In what follows, we make the convention that $\frac{\partial \psi}{ \partial x_k}^+= \frac{\partial \psi}{ \partial x_k}^-=\frac{\partial \psi}{ \partial x_k}$ out of $J_{\frac{\partial \psi}{\partial x_k}}$, so that the precise representative satisfies
  $$  \frac{\partial \psi}{ \partial x_k}= \frac 12 \Big(\frac{\partial \psi}{ \partial x_k}^++ \frac{\partial \psi}{ \partial x_k}^-\Big)$$
everywhere out of a $\H^{N-1}$ negligible set. This convention allows us to write for any $k \in \{1,\cdots,N\}$,
  $$  \Big(\frac{\partial \psi}{\partial x_k}^+ -\frac{\partial \psi}{\partial x_k}^-\Big) \langle \nu_{\frac{\partial \psi}{\partial x_k} }, \textstyle{\frac{\partial }{\partial x_i}}\rangle
    \H^{N-1}\res_{ J_{\frac{\partial \psi}{ \partial x_i}}} = \Big(\frac{\partial \psi}{\partial x_k}^+ -\frac{\partial \psi}{\partial x_k}^-\Big) \langle \nu_{\nabla^E \psi }, \textstyle{\frac{\partial }{\partial x_i}}\rangle   \H^{N-1}\res_{ J_{\nabla^E \psi}}. $$

Using these properties, we can now compare the jump parts of the derivatives. Starting with     
  \begin{eqnarray*}
  \sum_{1\leq k,s,\leq N} g^{ks} \frac{\partial \psi}{\partial x_s}\frac{\partial^2 \psi}{\partial x_k\partial x_i} & = & \sum_{1\leq k,s,\leq N} g^{ks} \frac{\partial \psi}{\partial x_s}\frac{\partial^2 \psi}{\partial x_i\partial x_k} \\
  														&= &  \sum_{1\leq k,s,\leq N} \frac{g^{ks}}{2} \Big(\frac{\partial \psi}{\partial x_s}^+ + \frac{\partial \psi}{\partial x_s}^-\Big) \Big(\frac{\partial \psi}{\partial x_k}^+ - \frac{\partial \psi}{\partial x_k}^-\Big) \langle \nu_{\frac{\partial \psi}{\partial x_k}, \frac{\partial}{\partial x_i}}\rangle \H^{N-1}\res_{J_{\frac{\partial \psi}{\partial x_k}}}\\
  														&=& \Big(\sum_{1\leq k,s,\leq N} \frac{g^{ks}}{2} \Big(\frac{\partial \psi}{\partial x_s}^+ + \frac{\partial \psi}{\partial x_s}^-\Big) \Big(\frac{\partial \psi}{\partial x_k}^+ - \frac{\partial \psi}{\partial x_k}^-\Big)\Big)  \langle \nu_{\nabla^E \psi, \frac{\partial}{\partial x_i}}\rangle\H^{N-1}\res_{ J_{\nabla^E \psi}}\\
													&=& \Big(\sum_{1\leq k,s,\leq N} \frac{g^{ks}}{2} \Big(\frac{\partial \psi}{\partial x_s}^+\frac{\partial \psi}{\partial x_k}^+ - \frac{\partial \psi}{\partial x_s}^-\frac{\partial \psi}{\partial x_k}^-\Big)\Big)  \langle \nu_{\nabla^E \psi, \frac{\partial}{\partial x_i}}\rangle \H^{N-1}\res_{J_{\nabla^E \psi}}.
													\end{eqnarray*}
  
  We now consider $\frac{1}{2} \sum_{1\leq k,s\leq N}  g^{ks} \frac{\partial}{\partial x_i}\Big(\frac{\partial \psi}{\partial x_s}\frac{\partial \psi}{\partial x_k}\Big) $. 
  $$ \begin{array}{lrl}
  \mbox{On }   J_{\frac{\partial \psi}{\partial x_k}} \cap J_{\frac{\partial \psi}{\partial x_s}}, &  \frac{\partial}{\partial x_i}\Big(\frac{\partial \psi}{\partial x_s}\frac{\partial \psi}{\partial x_k}\Big) &=\Big(\frac{\partial \psi}{\partial x_s}^+\frac{\partial \psi}{\partial x_k}^+ - \frac{\partial \psi}{\partial x_s}^-\frac{\partial \psi}{\partial x_k}^-\Big)  \langle \nu_{\nabla^E \psi, \frac{\partial}{\partial x_i}}\rangle\H^{N-1}\res_{ J_{\nabla^E \psi}}\\ 																																		 \mbox{On }   J_{\frac{\partial \psi}{\partial x_k}} \setminus J_{\frac{\partial \psi}{\partial x_s}}, &	 \frac{\partial}{\partial x_i}\Big(\frac{\partial \psi}{\partial x_s}\frac{\partial \psi}{\partial x_k}\Big) &=		\frac{\partial \psi}{\partial x_s}\Big(\frac{\partial \psi}{\partial x_k}^+ - \frac{\partial \psi}{\partial x_k}^-\Big)  \langle \nu_{\nabla^E \psi, \frac{\partial}{\partial x_i}}\rangle \H^{N-1}\res_{J_{\nabla^E \psi}}\\
  &	 &=		\frac{1}{2}\Big(\frac{\partial \psi}{\partial x_s}^+ + \frac{\partial \psi}{\partial x_s}^-\Big) \Big(\frac{\partial \psi}{\partial x_k}^+ - \frac{\partial \psi}{\partial x_k}^-\Big) \langle \nu_{\nabla^E \psi, \frac{\partial}{\partial x_i}}\rangle \H^{N-1}\res_{J_{\nabla^E \psi}	}\\			
  	 \mbox{On }    J_{\frac{\partial \psi}{\partial x_s}} \setminus J_{\frac{\partial \psi}{\partial x_k}}, &	 \frac{\partial}{\partial x_i}\Big(\frac{\partial \psi}{\partial x_s}\frac{\partial \psi}{\partial x_k}\Big) &=		\frac{\partial \psi}{\partial x_k}\Big(\frac{\partial \psi}{\partial x_s}^+ - \frac{\partial \psi}{\partial x_s}^-\Big)  \langle \nu_{\nabla^E \psi, \frac{\partial}{\partial x_i}}\rangle \H^{N-1}\res_{J_{\nabla^E \psi}}\\
  &	 &=		\frac{1}{2}\Big(\frac{\partial \psi}{\partial x_k}^+ + \frac{\partial \psi}{\partial x_k}^-\Big) \Big(\frac{\partial \psi}{\partial x_s}^+ - \frac{\partial \psi}{\partial x_s}^-\Big) \langle \nu_{\nabla^E \psi, \frac{\partial}{\partial x_i}}\rangle \H^{N-1}\res_{J_{\nabla^E \psi}}													 	
  					\end{array}$$
    
Consequently, we get thanks to the symmetry with respect to $k$ and $s$,
$$\frac{1}{2} \sum_{1\leq k,s\leq N}  g^{ks} \frac{\partial}{\partial x_i}\Big(\frac{\partial \psi}{\partial x_s}\frac{\partial \psi}{\partial x_k}\Big) =  \Big(\sum_{1\leq k,s,\leq N} \frac{g^{ks}}{2} \Big(\frac{\partial \psi}{\partial x_s}^+\frac{\partial \psi}{\partial x_k}^+ - \frac{\partial \psi}{\partial x_s}^-\frac{\partial \psi}{\partial x_k}^-\Big)\Big)  \langle \nu_{\nabla^E \psi, \frac{\partial}{\partial x_i}}\rangle J_{\nabla^E \psi}.$$
  
The equality is proved. 
\end{proof}

With the above result at our disposal, we can now establish the integration by parts formula involving the Hessian.

\begin{prop}[Integration by parts formula]\label{gamma2} Let $v$ (respectively $u$) be a $DC$ (resp. $DC_0$) function defined on an open subset $\Omega$ of $X^*$. Then, for any compactly supported Lipschitz function $\psi$ defined on $\Omega$, the following equality holds
$$ \int_\Omega \psi\, \hes \, v (\nabla u,\nabla u)  = - \int_\Omega \psi\, g(\nabla v,\nabla u) \Delta^g u   -\frac{1}{2} \int_\Omega \psi\, g(\nabla v, \nabla  |\nabla u|_g^2)  - \int_\Omega g(\nabla v, \nabla u) g(\nabla u, \nabla \psi)\,dv_g .$$
\end{prop} 

\begin{proof}
The claim follows from the integration by parts formula
\begin{equation}\label{IPP2}
\int_{\Omega}  \psi\, g(\nabla v,\nabla u)  \Delta^g u = - \int_{\Omega} g(\nabla u, \nabla (g(\nabla v,\nabla u) \psi ))
\end{equation}
proved below. Let us first explain how to complete the proof from this equality. First, notice that $ g(\nabla v,\nabla u)$ and $\psi$ are $\GBV$ functions on $\Omega$ and that $\psi$ has  no jump part in its derivative. Therefore according to Lemma \ref{lemmaproduct}, the Leibniz rule holds for these functions:
$$ \nabla (g(\nabla v,\nabla u) \psi )=\psi  \nabla (g(\nabla v,\nabla u)) +  g(\nabla v,\nabla u) \nabla \psi \H^N.$$
This yields
\begin{equation}\label{IPP3} 
 \int_{\Omega}   \psi\, g(\nabla v,\nabla u)  \Delta^g u= -  \int_{\Omega} \psi \,g(\nabla u, \nabla (g(\nabla v,\nabla u)) - \int_{\Omega} g(\nabla v,\nabla u)g(\nabla u,\nabla \psi)\,d\H^N.
\end{equation}
Now we can rewrite the first term on the right-hand side as
\begin{eqnarray}\label{IPP4}
 \int_{\Omega} \psi \, g(\nabla u, \nabla (g(\nabla v,\nabla u))&= &\int_{\Omega} \psi \,D (dv (\nabla u)) (\nabla u) \nonumber  \\
 																								&=&\int_{\Omega} \psi \, \hes v (\nabla u,\nabla u)+ \int_{\Omega} \psi  \,g(\nabla v, D_{\nabla u} \nabla u).\nonumber\\
 																								&=& \int_{\Omega} \psi \, \hes v (\nabla u,\nabla u)+\frac{1}{2}\int_{\Omega} \psi  \,g(\nabla v,\nabla |\nabla u|_g^2)
\end{eqnarray} 
where we used Proposition~\ref{FormHessian} to get the second equality, and Proposition \ref{geod} to get the last one. Inserting (\ref{IPP4}) into (\ref{IPP3}) gives the result.

It remains to prove (\ref{IPP2}). Reasoning as in the proof of Proposition 5.9, we can further assume that $\psi$ is supported in the domain of a chart, thus it suffices to prove the result in local coordinates defined on an open subset of $\R^N$ that we also called $\Omega$ for simplicity.  Let $(\rho_{\ep})_{\ep>0}$ be a family of standard radial mollifiers. Let us set $h=g(\nabla v,\nabla u) $ and note that $h\in C_{w}(\Omega)$. Therefore, $h*\rho_{\ep}\rightarrow h$ pointwise out of a $\sigma$-finite set with respect to $\H^{N-1}$. The Lebesgue dominated convergence theorem then yields
\begin{eqnarray}
\int_{\Omega} \psi\, g(\nabla v,\nabla u)  \Delta^g u &=& \lim_{\ep \downarrow 0} \int_{\Omega} \psi\, h*\rho_{\ep} \, \Delta^g u \nonumber \\
																				&=& \lim_{\ep \downarrow 0} \left( \int_{\Omega} \psi\,g(\nabla  h*\rho_{\ep}, \nabla  u)dv_g + \int_{\Omega} h*\rho_{\ep} \,g(\nabla \psi , \nabla u )dv_g\right)
\end{eqnarray}
Since $|g(\nabla \psi,\nabla u)|$ is compactly supported and bounded, it is clear that
$$ \lim_{\ep \downarrow 0}  \int_{\Omega} h*\rho_{\ep} \,g(\nabla \psi , \nabla u )dv_g = \int_{\Omega}  g(\nabla v,\nabla u)\,g(\nabla \psi , \nabla u )dv_g.$$
We prove the convergence of the remaining integral thanks to Proposition \ref{teclemma}. Let us check that the hypotheses of Proposition \ref{teclemma} are satisfied. Using the coordinate system $(x_i)_{1\leq i\leq N}$, we get
$$ g(\nabla  h*\rho_{\ep}, \nabla  u) =\sum_{i=1}^N g^{ij} \frac{\partial u}{\partial x_i} \psi \,\frac{\partial (h*\rho_{\ep})}{\partial x_j}.$$
By assumption, $g^{ij} \frac{\partial u}{\partial x_i} \psi  \in C_{w,0}(\Omega)$ while $\frac{\partial h}{\partial x_j} \in \GM(\Omega)$. Moreover the convergence $\frac{\partial (h*\rho_{\ep})}{\partial x_j} \rightharpoonup \frac{\partial h}{\partial x_j}$ in the duality with $C_c(\Omega)$  and (\ref{eq:stimamax}) hold (see for instance \cite[Theorem 2.2]{AFP}). Therefore
 $$\lim_{\ep \downarrow 0} \int_{\Omega} \psi\,g(\nabla  h*\rho_{\ep}, \nabla  u) = \int_{\Omega} \psi\,g(\nabla  h, \nabla  u)$$
 and the proof of (\ref{IPP2}) is complete. 
\end{proof}
%%%%%%%%%%%%%%%%%%%%%%%%%%%v%%%%%%%%%%%%%%%%%%%%%%%%%%%
%%%%%%%%%%%%%%%%%%%%%%%%%%%%%%%%%%%%%%%%%%%%%%%%%%%%%%

\section{Alexandrov spaces}

In this part we prove that our results apply to an open dense subset $X^*$ of a finite dimensional Alexandrov space $(X,d)$ with curvature bounded from below. For an introduction to these spaces we refer to the book \cite{bbi}, see also \cite{AB} where all the notions below are discussed with further details.

\subsection{The quasiregular set of an Alexandrov space is a DC Riemannian manifold}
\begin{defi}[Quasiregular set $X^*$]Let $(X,d)$ be an $N$-dimensional Alexandrov space with curvature bounded below by $k$. Given $\delta>0$, a point $x$ is $\delta$-regular if there exist 
$N$ pairs of points $(p_1,q_1),\cdots,(p_N,q_N)$ such that 
$$
\begin{cases}
 \tilde{\angle} p_ixp_j > \frac{\pi}{2} -\delta & \text{ for all $i \neq j$,}\\
 \tilde{\angle}  p_i xq_i > \pi -\delta             &  \text{ for all $i$}
\end{cases}
$$
where $\tilde{\angle} xyz$ stands for the angle at $y$ of a comparison triangle in the space form of curvature $k$. The collection of pairs $(p_1,q_1)\cdots (p_N,q_N)$ is called a  $\delta$-strainer (at $x$). By lower semicontinuity of angles  with respect to $x$, the set of $\delta$-regular points is open. 
We shall denote by $X^*$ the set of quasiregular points, namely the set of $\delta_N$-regular points for a fixed $\delta_N \ll 1/N$, omitting the dependence
of $\delta_N$ for simplicity of notation.
\end{defi}
\begin{remark} It is proved in \cite{bbi} (see also the original paper \cite{bgp}) that $\delta_N = \frac{1}{100N}$ is a suitable choice. Namely, for such a choice, the open set $X^*$ is a dense subset of $X$ and a Lipschitz manifold. Note that 
it can be proved that these properties remain true for the set of $\delta$-regular points where $0<\delta \leq \delta_N$ (see \cite[Corollary 10.8.24]{bbi}). 
\end{remark}
\begin{defi}[$Reg(X)$ and $Sing(X)$] The set $Reg(X)$ of regular points is the set of points whose tangent cone is isometric to $N-$dimensional Euclidean space. 
Equivalently, it is the set of points which are $\delta$-regular for any $\delta>0$. The complement in $X$ of $Reg(X)$ is called singular set and denoted 
by $Sing(X)$. 
\end{defi} 
 
 Later, it was proved that actually $X^*$ is a $DC_0$ Riemannian manifold according to our terminology.
\begin{thm}[]\label{diff}
Given $(X,d)$ a $N$-dimensional Alexandrov space with curvature bounded from below, its quasiregular set $X^*$ is a $DC_0$ Riemannian manifold with singular set 
$\Si= X^*\cap Sing(X)$.
 More precisely, there exist a Riemannian metric $g$ and locally biLipschitz charts $\tphi: U_{\tphi} \rightarrow \mathbb{R}^n$ defined by the formula
\[  \tphi= (\hd_{p_1},\cdots,\hd_{p_n})\] 
where   
\[\hd_{p_i}(x)= \fint_{B(p_i,\ep_i)} d_{z_i}(x) \,d\H^N (z_i)
,\] 
such that $X^*$ is a $DC_0$ Riemannian manifold. Moreover, the components $g_{ij}$ of the Riemannian metric when read in a chart,  belong to $\GBV_0$ and 
satisfy
\begin{equation}\label{rmetbounds}
\frac{1}{c}\|p\|_2^2\leq \sum_{i,j} g_{ij}(x)p_ip_j\leq c\|p\|_2^2\quad\text{for all $p\in\R^N$, for $x\in\hat\phi(U'\setminus\Si)$},
\end{equation}
with $c=c(U')>0$ for all $U'\Subset U$. Last, the components $g_{ij}$ are differentiable Lebesgue almost everywhere.
\end{thm}

\begin{proof} Property (a) of the definition of $DC_0$ Riemannian manifold was proved by Otsu and Shioya \cite{OS} (with the exception of the $DC_0$ character of the transition maps which is due to Perelman \cite{DC}) as well as the existence of a Riemannian metric $g$ whose components $g_{ij}$ are in $\C_{w,0}$ out of $\Si=Sing(X)$. The BV character as well as the differentiability property of $g_{ij}$ was proved by Perelman in \cite{DC}. The estimate (\ref{rmetbounds}) is a consequence of results proved in these papers but is not properly stated as such. Let us give a proof of it for the sake of completeness. In the following $(\cdot,\cdot)_x$ stands for the inner product defined on the tangent cone $K_x$ (note that, strictly speaking, it is an inner product only when $x\in Reg(X)$ and in this case it coincides with $g$ at $x$) by the formula
$$ (u,v)_x= |u|\,|v| \cos \angle u,v.$$

Now, any $f_i:=\hd_{p_i}$ is a semiconcave function, thus it admits directional derivative (see \cite{SCPet}) along any tangent vector $u \in K_x$. Moreover, the directional derivative satisfies for all $x \in U_{\tphi}$ and $u\in K_x$
$$ f_i'(x,u) \leq (\xi_i(x),u)_x$$
where $\xi_i(x)$ is the ``gradient'' of $f_i$ at $x$ (see \cite{SCPet} for more details). Using the definition of $f_i$ as an average of distance functions, one can prove that 
$$ \text{for all } u \in K_x, \quad f_i'(x,u) = (\xi_i(x),u)_x$$
whenever $x \in Reg(X)$.

Indeed, given a unitary geodesic $\gamma$ starting at $x$, the first variation formula (and Lebesgue's dominated convergence theorem) yields
$$ f_i(\gamma(s)) = f_i(x) +s \fint_{B(p_i,\ep_i)} \cos \angle \gamma'(0), \uparrow_x^{p} d\H^N(p) + o_x(s)$$
where $\uparrow_x^{p}$ stands for the unit direction of the (unique for $\H^N$-a.e. $p \in B(p_i,\ep_i)$) geodesic from $x$ to $p$. See \cite{OS} for more details. Consequently, 
$$\fint_{B(p_i,\ep_i)} (\gamma'(0), \uparrow_x^{p})_x \,d\H^N(p) = \fint_{B(p_i,\ep_i)} \cos \angle \gamma'(0), \uparrow_x^{p} \,d\H^N(p) = f_i'(x, \gamma'(0)) \leq (\xi_i(x),\gamma'(0))_x$$
and, by density of the geodesic directions $\gamma'(0)$ in the space of directions $\Sigma_x(X)$ (the unit sphere of $K_x$), 
the above formula holds for any $u \in K_x$. Since the extreme terms are both linear with respect to $u$ when $x \in Reg(X)$, we do have equality in the inequality above whenever $x \in Reg(X)$.

Let now $U'\Subset U$. 
Using this equality for all $i \in \{1,\cdots,N\}$ and the fact that ${\tphi}$ is a locally biLipschitz map, namely the existence of $c=c(U')>0$ satisfying
$$ \frac{1}{c} \,d(x,y) \leq ||\tphi(x)-\tphi(y)||_2 \leq c \,d(x,y)\qquad\forall x,\,y\in U', $$
we get, for all $x \in Reg(X)\cap U'$, $u \in K_x$
  \begin{equation}\label{eq:metesti}
 \frac{1}{c} \leq ||((\xi_i(x),u)_x)_{1\leq i\leq N}||_2 \leq c,
 \end{equation}
since $\tphi'(x,u) = ((\xi_i(x),u)_x)_{1\leq i\leq N}$. We also infer from the above formula that $(\xi_i(x))_{1\leq i\leq N}$ forms a basis of $K_x$ whenever $x$ is a regular point. In particular for all $p \in \R^N$ we can find $u_p\in K_x$ such that
 $$ \tphi'(x,u_p) =p.$$
 Then, (\ref{eq:metesti}) yields
 $$   \frac{1}{c} \sqrt{(u_p,u_p)_x}\leq ||p||_2 \leq c\sqrt{(u_p,u_p)_x}$$
by homogeneity and the proof is complete. 
 \end{proof}

%%%%%%%%%%%%%%%%%%%%%%%%%%%%%%%%%%%%%%

\subsection{Almost everywhere second order Taylor expansion for $DC$ functions}

In this part, following Perelman's ideas \cite{DC}, we fully extend to the setting of Alexandrov spaces the classical result of Alexandrov on the existence of second order Taylor expansion for a convex function at Lebesgue almost every point. To be more specific, our goal is to prove the following result.\par 

 \begin{prop}\label{Taylor}
Let $f$ be a $DC$ function defined on $\Omega\subset X^*$. Then, for $\H^N$-a.e. point $x$,

$$f(y) = f(x) + |xy| df(x)(\uparrow_x^y) +\frac{1}{2} |xy|^2\hes^{ac} f (\uparrow_x^y,\uparrow_x^y) + o(|xy|^2)$$
where $|xy|= d(x,y)$ and $\uparrow_x^y$ denotes the direction at $x$ of an arbitrary geodesic from $x$ to $y$.  
\end{prop}
\begin{remark}
We don't know whether this result holds on any $DC_0$ Riemannian manifold. The proof we give below follows the arguments given by Perelman in \cite[Proposition p7]{DC} with more details added. Indeed, this result is of local nature, so there is no need to have a well-defined global object $\hes$ to consider this question. 
However, as stated, it shows that our notion of $\hes$ gives geometric informations on the function.
\end{remark}

To prove the result above, we need first to introduce normal coordinates.

\subsubsection{Normal coordinates}

In this part, following Perelman's ideas, we establish some properties of normal coordinates around ``good'' points of Alexandrov space. We think these results could be useful in other contexts, so we state them independently of Proposition~\ref{Taylor}.\par 

\begin{remark}\label{goodp} Note that by definition of a $DC_0$ Riemannian manifold, if the metric components are differentiable at point $F(x)$, $x\in Reg(X)$, when read in a chart $F$  then this is also true in any other chart whose domain contains $x$. The same property holds for the second order Taylor expansion of any $DC$ function defined on an open subset of $X^*$. Consequently, it makes sense to say that the metric $g$ (or a mere function) is differentiable at a point $x \in Reg(X)$. According to Theorem \ref{diff}, $g$ is differentiable at $\H^N$-almost every point.
\end{remark}

\begin{prop}\label{normal} Consider a point $x\in Reg(X)$ where the metric is differentiable. Then, there exists a normal coordinate system $N$ defined on a neighborhood of $x$ and compatible with the $DC_0$ structure (i.e. it makes a $DC_0$ mapping when composed with the inverse of any chart). Namely, in the coordinate system $N$, we have the following properties:

\begin{enumerate}[(a)]
\item $|g_{ij}(N(y)) -\delta_{ij}| = o(|xy|)$ for all $y \in Reg(X)$;

\vspace{0.2cm}
\item $|\,|yz| -||N(y)N(z)||_2| = o(\max\{|xy|,|xz|\})|yz|$;
\vspace{0.2cm}
\item $\lim_{y \rightarrow x} \angle_E (\uparrow_{N(x)}^{N(y)} N(x)N(y)) = 0,$
where $\angle_E$ stands for the standard Euclidean angle and $ \uparrow_{N(x)}^{N(y)}$ stands for the image through $N$ of the direction $\uparrow_x^y$ of a geodesic from $x$ to $y$;
\vspace{0.2cm}
\item$ | \tilde{\angle } yxz - \angle_E N(y)N(x)N(z)| = o(|yz|)  $
when all the angles of $N(x)N(y)N(z)$ are bounded away from $0$. (The same estimate holds if we consider a comparison triangle in $\R^2$ and replace $\tilde{\angle } yxz$ by its Euclidean counterpart).
\end{enumerate}
\end{prop}

\begin{proof}
In order to define $N$, we consider the bilinear form $(\tilde{g}_{ij})_{i,j}$ defined in the chart $F$ by the formula
$$ \tilde{g}_{ij}(\bar{z})= g_{ij}(\bar{x})  + \sum_{k=1}^N (\bar{z}-\bar{x})_k \frac{\partial g_{ij}}{\partial x_k}(\bar{x})$$
being $g_{ij}$ the components of the metric read in $F$ and $\bar x=F(x)$, $\bar{z}=F(z)$. For $z$ close to $x$, $(\tilde{g}_{ij})_{i,j}$ is a smooth
Riemannian metric whose value at $\bar{x}$ coincides with that of $(g_{ij})_{i,j}$  and the same property is true for the first derivatives
of $\tilde{g}_{ij}$. Therefore, we can find a smooth normal coordinate system for $\tilde{g}_{ij}$ and we call $N$ the coordinate system obtained by composing it with $F$.  
Note that in particular $N$ is compatible with the $DC_0$ structure. By definition of a normal coordinate system, the metric components of $\tilde{g}$ read in this normal coordinate system satisfy (we keep the same notation for simplicity)
$$ g_{ij}(\bar{x}) =\delta_{ij} \mbox{    and    }    \frac{\partial g_{ij}}{\partial x_k}(\bar{x})=0 \mbox{       for all } i,j,k.$$
Therefore, by construction, the same properties hold for $g$ at $N(x)$ when read in the chart $N$. The first item then immediately follows. Let us prove the second one.\par

Set $\gamma$ a constant speed geodesic defined on $[0,1]$ from $y$ to $z$. Note that by density, it is sufficient to prove (b) for $y,\,z \in Reg(X)$. According to \cite{OS}, $F(\gamma)$ belongs to $C^1((0,1))$, therefore $\sigma := N(\gamma)$ belongs to $C^1((0,1))$ as well. More generally, Otsu and Shioya proved \cite[Proposition 6.2]{OS} that the length $L(\theta)$ of any $C^1$ curve $\theta: [0,1] \rightarrow Reg(X)$ satisfies
\begin{equation}\label{long}
 L(\theta) = \int_0^1 \sqrt{g_{\theta(s)} (\theta'(s),\theta'(s))} \,ds.
 \end{equation}

Now, we infer from (a) the following estimate
$$ | \,||\sigma'(s)||_{2} - |\sigma'(s)|_g| = \frac{| \,||\sigma'(s)||_{2}^2 - |\sigma'(s)|_g^2|}{ ||\sigma'(s)||_{2} + |\sigma'(s)|_g} \leq o(\max\{|xy|,|xz\})|yz|$$
where we use the fact that $N$ is in particular biLipschitz (thus $o(|\cdot|)=o(||\cdot||_2)$). Integrating the inequality above leads to
$$ ||N(y)N(z)||_2 \leq L_{||\cdot||_2}(\sigma) \leq |yz| + o(\max\{|xy|,|xz\})|yz|$$
where $L_{||\cdot||_2}(\sigma) $ stands for the length of $\sigma$ with respect to Euclidean metric. The converse inequality is proved in the same way, up to replace the segment $[N(y)N(z)]$ by an arbitrary close and parallel segment of same (Euclidean) length. Indeed, first notice that the Coarea inequality 
(see for instance \cite[Theorem~2.10.25]{federer69}) yields
$$ 0=\H^{N-1} (Sing(X) \cap R) \geq \int_{\R^{N-1}} \H^{0}(Sing(X) \cap p^{-1}(z))dz,$$
being $R$ a cylinder with axis $[N(y)N(z)]$, and $p$ the projection parallel to $N(y)N(z)$. Therefore, by continuity, we can assume without loss of generality that $[N(y)N(z)]\cap Sing(X)= \emptyset$ so that the equality (\ref{long}) holds and the previous argument applies.\par

We now prove (c). For any integer $k\geq 0$, set $m_k:= \gamma (\frac{1}{2^k})$, being $\gamma$ a constant speed geodesic from $x$ to $y$, parameterized on $[0,1]$ (so that $m_0=y$, $m_1$ is a midpoint of $x$ and $y$ and so on). Note that for any vector $u$, 
$$\lim_{k \rightarrow + \infty} \angle_E (N(x)N(m_k), u) = \angle_E (\uparrow_{N(x)}^{N(y)}, u).$$
Let us fix $k\geq1$. Elementary computations together with (b) lead to
$$ \cos \angle_E N(x)N(m_k)N(m_{k-1}) +1 \leq o(|xm_{k-1}|).$$
Therefore, for $y$ sufficiently close to $x$, we get
$$ \angle_E N(m_k)N(x)N(m_{k-1}) \leq \sqrt{|xm_{k-1}|} \ep(|xm_{k-1}|)$$
where $\lim_{t\rightarrow 0} \ep(t)=0$. By definition, $|xm_{k-1}| =\frac{|xy|}{2^{k-1}}$. Moreover, up to reducing $|xy|$, we can assume that
$|\ep(s)| \leq 1$ on $[0, |xy|]$. Consequently, we get 

$$ \angle_{E} ( \uparrow_{N(x)}^{N(y)}, N(x)N(y)) \leq \sqrt{|xy|}\, \sum_{k=0}^{\infty} \frac{1}{2^{\frac{k}{2}}} \underset{y\rightarrow x}{\longrightarrow} 0$$
and the proof of (c) is complete.\par 

It remains to prove (d). First, notice that since the angles of $N(x)N(y)N(z)$ are assumed to be bounded away from $0$, the error terms $o(|N(x)N(y)|), o(|N(x)N(z)|),o(|N(z)N(y)|)$ are all the same and the property remains true if we use $|xy|,|xz|,|yz|$ instead since $N$ is locally a biLipschitz map. Second, since these distances are assumed to be small, the law of cosines in the space form of curvature $k$ implies that
\begin{equation}\label{cosine}
 |\cos \tilde{\angle} yxz -\cos \tilde {\angle}_0 yxz| \leq O(|xy|^2)
 \end{equation}
being $\tilde{\angle}_0 yxz$ the angle at $x$ of a comparison triangle in $\R^2$. Now, using the Euclidean law of cosines for both $N(x)N(y)N(z)$ and a Euclidean comparison triangle of $xyz$,  and the estimate (2) (recall that $o(|xy|)=o(|xz|)=o(|yz|)$) gives $|\cos \angle_E N(y)N(x)N(z) -\cos \tilde{\angle}_0 yxz| = o(|xy|^2)$. Finally, combining this with (\ref{cosine}) yields
$$ |\cos \tilde{\angle} yxz -\cos \angle_E N(y)N(x)N(z)| \leq O(|xy|^2).$$
Using again the fact that the angles of $N(x)N(y)N(z)$ are bounded away from $0$ (and $\pi$ as well) allows us to conclude.
\end{proof}

\subsubsection{Proof of Proposition \ref{Taylor}}

We will use the notation introduced in Proposition \ref{normal}. By definition, a $DC_0$ Riemannian manifold can be covered by countably many domains of (biLipschitz) charts, therefore it suffices to prove the result locally. Also, recall Theorem \ref{diff} asserting that the metric $g$ is differentiable $\H^N$-almost everywhere (see Remark \ref{goodp}). 
Therefore, according to Proposition \ref{normal}, it suffices to prove the result at every point $x \in Reg(X)$ where $g$ is differentiable and the first and second derivatives of $f$ do exist, and we can use a normal coordinate system $N$ around $x$ to proceed. (Strictly speaking, to guarantee that it is indeed $\hes^{ac} f$ that appears in the Taylor expansion, we have to discard a $\H^N$-negligible set of points where the second derivatives of $f$ and the derivatives of the metric are not approximately continuous, and $|\hes^{s} f| (B(x,r)= o(r^N)$ is not satisfied. For more details, we refer to the proof of Alexandrov's theorem p242 in \cite{evans} or to the proof of Theorem 3.83 in \cite{AFP}.) For the rest of the proof, we fix such a point $x$ and a normal coordinate system $N$ around $x$.
We have (writing $f$ instead of $f\circ N^{-1}$)
$$ f (y) = f(x) + d_{N(x)}f (N(x)N(y)) + \frac{1}{2} D^2_{N(x)} f (N(x)N(y), N(x)N(y)) + o (|N(x)N(y)|^2).$$

Note that in a normal coordinate system around $x$, $ D^2_{N(x)} f= \hes_{N(x)} f$. Therefore, in view of Proposition~\ref{normal}, we are done if we can prove
\begin{equation}\label{angle2}
 \angle_E  \uparrow_{N(x)}^{N(y)}, N(x)N(y) = o(|xy|).
 \end{equation}
  
To this aim, we first claim that
$$  | \angle_E  \uparrow_{N(x)}^{N(y)}, \uparrow_{N(x)}^{N(z)} - \angle_E N(y)N(x)N(z)|= | \angle  \uparrow_x^y, \uparrow_x^z - \angle_E N(y)N(x)N(z)| = o(|yz|)$$
when all the angles of $N(x)N(y)N(z)$ are bounded away from $0$.\par 

The first equality follows from the definition of normal coordinate system.To prove the second one, take a point $p$ such that $N(p)$ is the plane $N(y)N(x)N(z)$, $N(x)$ is contained in the triangle $N(p)N(y)N(z)$ and all the angles formed by these four points are bounded away from zero. Recall that for any triple $x,\,y,\,z$, one has 
$\angle  \uparrow_x^y, \uparrow_x^z \geq  \tilde{\angle}  yxz $ by definition of Alexandrov space. Property (d) in Proposition~\ref{normal} then gives
\begin{eqnarray*}
&& \angle  \uparrow_x^y, \uparrow_x^z + \angle  \uparrow_x^z, \uparrow_x^p + \angle  \uparrow_x^p, \uparrow_x^y \\&\geq&  
 \angle_E N(y)N(x)N(z) +  \angle_E N(z)N(x)N(p) + \angle_E N(p)N(x)N(y) + o(|yz|) \\&=& 2\pi + o(|yz|).
  \end{eqnarray*}

On the other hand, $ \angle  \uparrow_x^y, \uparrow_x^z + \angle  \uparrow_x^z, \uparrow_x^p + \angle  \uparrow_x^p, \uparrow_x^y \leq 2\pi$ (this is a consequence of the quadruple condition, see for instance \cite{bbi}) and the claim is proved.\par

The proof of (\ref{angle2}) is by contradiction. Take a point $y_1$ such that $|xy_1|= {|xy|}/{2}$  and the direction $\uparrow_{N(x)}^{N(y)}$ is between $N(x)N(y)$ and $N(x)N(y_1)$. In particular, this gives
\begin{eqnarray*}
\angle_E \uparrow_{N(x)}^{N(y_1)}, N(x)N(y_1) &\geq& \angle_E \uparrow_{N(x)}^{N(y_1)}, \uparrow_{N(x)}^{N(y)} - \angle_E \uparrow_{N(x)}^{N(y)}, N(x)N(y_1) \\
																	 &\geq &  \angle_E \uparrow_{N(x)}^{N(y_1)}, \uparrow_{N(x)}^{N(y)} - \angle_E N(y)N(x)N(y_1) + \angle_E \uparrow_{N(x)}^{N(y)}, N(x)N(y) \\
																	  & \geq & o(|xy|) + \angle \uparrow_{N(x)}^{N(y)}, N(x)N(y).
\end{eqnarray*}
Thus if (\ref{angle2}) were false, then we could construct a sequence $y_k \rightarrow x$ with $\angle \uparrow_{N(x)}^{N(y_k)}, N(x)N(y_k) $ bounded away from $0$ and this would contradict (c) in Proposition~\ref{normal}.

 %%%%%%%%%%%%%%%%%%%%%%%%%%%%%%%%%%
%%%%%%%%%%%%%%%%%%%%%%%%%%%%%%%%%%
 %%%%%%%%%%%%%%%%%%%%%%%%%%%%%%%%%%%%%%%%%%%%
 %%%%%%%%%%%%%%%%%%%%%%%%%%%%%%%%%%%%%%%%%%%%

\section*{Appendix: The transformation law of the Christoffel symbols}

The goal of this part is to prove, in our nonsmooth setup, the following classical formula (\ref{Chris2}) for the transformation of Christoffel symbols. 
\begin{equation}\label{Chris2}
\Gamma_{ab}^k = \sum_{\theta} \frac{\partial F_k ^{-1}}{\partial y_{\theta}} \circ F \left( \sum_{m,t } \frac{\partial F_m}{\partial x_a} \frac{\partial F_t}{\partial x_b} F^* (\tilde{\Gamma}_{mt}^{\theta})\right) \\
+ \sum_{\theta} \frac{\partial  F_k^{-1}}{\partial y_{\theta}} \circ F \,\frac{\partial ^2 F_{\theta}}{\partial x_a \partial x_b}, 
\end{equation}
used in the proof of Proposition~\ref{compa} (see (\ref{TL1})), where $F:\hat{U}\to\hat{V}$ is a transition map
and $\Gamma$, $\tilde\Gamma$ denote the Christoffel symbols in $\hat{U}$ and $\hat{V}$ respectively.

We have not been able to locate a proof of (\ref{Chris2}) in the literature which is only based on (\ref{metriccoor}). So we provide such a proof for the sake of completeness. The tools we need to do so are basically those obtained in Corollary~\ref{corteclem}. Let us recall that, by definition, 
$$ 2\, \Gamma_{ij}^k = \sum_l g^{kl}\left( \frac{\partial g_{li}}{\partial x_j} + \frac{\partial g_{lj}}{\partial x_i} - \frac{\partial g_{ij} }{\partial x_l}\right).$$
For the sake of clarity, we shall use as much as possible matrix products to write down the formulas. 
For instance, the above formula can be rephrased in the following way
$$ 2\,\Gamma^k_{ij} = \left(G^{-1}\frac{\partial G}{\partial x_j}\right)_{ki}
+ \left(G^{-1}\frac{\partial G}{\partial x_i}\right)_{kj} - \sum_l g^{kl}\frac{\partial g_{ij} }{\partial x_l},$$
where $G$ denotes the matrix $(g_{ij})$, i.e. the metric in the coordinate system of $\hat{U}$.

We set $\Delta_j = G^{-1}\frac{\partial G}{\partial x_j}$. Denoting by $\tilde G=(\tilde g)_{ij}$ the metric
in the coordinate system of $\hat{V}$, we shall use the fact that (denoting by $^t M$ matrix transposition)
\begin{equation}\label{metriccoor}
\begin{array}{ccl}
G & = &  ^t\big( dF\big) \big(\tilde{G} \circ F \big) \left(dF\right)\\
&\mbox{and}&\\
G^{-1} & = &  \left( d_{F(\cdot)} F^{-1}\right) \big( \tilde{G}^{-1} \circ F\big) ^t\left(d_{F(\cdot) } F^{-1}\right).
\end{array}
\end{equation}
Thus,
\begin{equation}\label{dermetriccoor}
 \big( \frac{\partial G}{\partial x_j}\big) = ^t\big( \frac{\partial}{\partial x_j} (dF)\big)\big(\tilde{G} \circ F\big) \big(dF\big) + ^t\big( dF\big) \big(\frac{\partial}{\partial x_j} (\tilde{G}\circ F) \big) \big( dF\big) + ^t\big( dF\big) \big( \tilde{G} \circ F\big) \big(\frac{\partial}{\partial x_j} (dF)\big).
 \end{equation}
 Accordingly, we decompose $\Delta_j$ into three terms:
$$ \Delta_j = A_j + B_j +C_j,$$
where
\begin{equation}\label{expres}
 \left\{
 \begin{array}{rcl}
 \vspace{0.3cm} 
A_j & =&  \big( d_{F(\cdot)} F^{-1}\big) \big(\tilde{G}^{-1}\circ F \big)^t \big( d_{F(\cdot)}F^{-1}\big)^t\big( \frac{\partial}{ \partial x_j} (dF) \big) \big( \tilde{G}\circ F\big) \big( dF\big) \\
\vspace{0.3cm}
B_j & = & \big(  d_{F(\cdot)} F^{-1}\big) \big(\tilde{G}^{-1}\circ F \big)\big( \frac{\partial}{ \partial x_j }(\tilde{G} \circ F) \big) \big( dF\big)\\
C_j & =& \big( d_{F(\cdot)}F^{-1}\big) \big( \frac{\partial }{\partial x_j} (dF) \big).
\end{array}
\right.
\end{equation}

We first consider $C_j$ and we compute $(C_j)_{ki}$:
$$ (C_j)_{ki}= \sum_s \frac{\partial F_k^{-1}}{\partial y_s} \circ F \frac{\partial^2 F_s}{\partial x_j \partial x_i }.$$
Using the symmetry of the second distributional derivatives, we get 
$$  (C_j)_{ki}= (C_i)_{kj}= \sum_s \frac{\partial F_k^{-1}}{\partial y_s} \circ F\, \frac{\partial^2 F_s}{\partial x_j \partial x_i }.$$

Now, we treat the term $B_j$. We shall use that for $h \in\GBV_0$ (see (\ref{compoder}) in Corollary~\ref{corteclem}), the chain rule formula holds in the following form
\begin{equation}\label{crf}
 \frac{\partial }{\partial x_i}  ( h \circ F)= \sum_{s=1}^p \frac{\partial F_s}{\partial x_i} F^* \big( \frac{\partial h}{\partial y_s}\big).
 \end{equation}
Thus,
$$ (B_j)_{ki} = \sum_m \frac{\partial F_m}{\partial x_j}\Big[\big(d_{F(\cdot)} F^{-1}\big)\big(\tilde{G} ^{-1} \circ F \big) \big( F^* \big(\frac{\partial \tilde{G}}{\partial y_m}\big)\big) \big( dF\big) \Big]_{ki},$$
and we get by expanding this expression
$$   (B_j)_{ki} = \sum_{m,s,t} \frac{\partial F_m}{\partial x_j} \frac{\partial F_k^{-1}}{\partial y_s} \circ F \left[ \big( \tilde{G} ^{-1} \circ F \big) \big(F^*\big(\frac{\partial \tilde{G}}{\partial y_m}\big) \big)\right]_{st} \frac{\partial F_t}{\partial x_i}.$$

Similarly,
$$   (B_i)_{kj} = \sum_{t,s,m} \frac{\partial F_t}{\partial x_i} \frac{\partial F_k^{-1}}{\partial y_s} \circ F \left[ \left( \tilde{G} ^{-1} \circ F \right) \left(F^*\left(\frac{\partial \tilde{G}}{\partial y_t}\right) \right)\right]_{sm} \frac{\partial F_m}{\partial x_j}.$$
By adding the two formulas above, we get
$$  (B_j)_{ki} +(B_i)_{kj} = \sum_{m,s,t} \frac{\partial F_m}{\partial x_j} \frac{\partial F_t}{\partial x_i} \frac{\partial F_k^{-1}}{\partial y_s} \circ F\, 
F^* \left[ \left( \tilde{G} ^{-1} \frac{\partial \tilde{G}}{\partial y_m}\right)_{st} + \left( \tilde{G} ^{-1}\frac{\partial \tilde{G}}{\partial y_t}\right)_{sm} \right],$$
where we use the fact that for $\alpha \in\GBV_0$, $\mu\in\GM$, $F^*(\alpha \mu) = (\alpha \circ F) \, F^*(\mu)$ (see Corollary~\ref{corteclem}).

Let us notice that 
$$\begin{array}{rcl}
F^* \left( \big( \tilde{G} ^{-1}  {\displaystyle\frac{\partial \tilde{G}}{\partial y_m}}\big)_{st} + \big( \tilde{G} ^{-1}  
\displaystyle{\frac{\partial \tilde{G}}{\partial y_t}}\big)_{sm} \right) 
&=&    F^*\left( 2 \tilde{\Gamma}_{mt}^s + \sum_{\theta} \tilde{g}^{s\theta} 
\displaystyle{\frac{\partial \tilde{g}_{mt}}{\partial y_{\theta}}}\right) \\
&=& F^*\left( 2 \tilde{\Gamma}_{mt}^s\right) + \sum_{\theta} \tilde{g}^{s\theta} \circ F \, F^* \left(\displaystyle{\frac{\partial \tilde{g}_{mt}}{\partial y_{\theta}}}\right). 
\end{array}$$

Therefore, if we combine all the equalities above, we get 
\begin{multline*}
2 \big( \Gamma_{\cdot \cdot}^k\big)_{ij} = 2 \sum_s \frac{\partial F_k^{-1}}{\partial y_s} \circ F \frac{\partial^2 F_s}{\partial x_j \partial x_i } 
+ 2       \sum_{m,s,t} \frac{\partial F_m}{\partial x_j} \frac{\partial F_t}{\partial x_i}  \frac{\partial F_k^{-1}}{\partial y_s} \circ F\, F^* \left( \tilde{\Gamma}_{mt}^s \right) +\\
           \sum_{m,s,t,\theta}    \frac{\partial F_m}{\partial x_j} \frac{\partial F_t}{\partial x_i} \frac{\partial F_k^{-1}}{\partial y_s} \circ F\,     \tilde{g}^{s\theta} \circ F \, F^* \left(\frac{\partial \tilde{g}_{mt}}{\partial y_{\theta}}\right)  
    + \big(A_j\big)_{ki} + \big(A_i\big)_{kj} - \sum_u g^{ku}\frac{\partial g_{ij}}{\partial x_u}.
\end{multline*}
         
The proof is then complete if we prove that 
 \begin{equation}\label{topr}
    \sum_{m,s,t,\theta}    \frac{\partial F_m}{\partial x_j} \frac{\partial F_t}{\partial x_i} \frac{\partial F_k^{-1}}{\partial y_s} \circ F\,     \tilde{g}^{s\theta} \circ F \, F^* \left(\frac{\partial \tilde{g}_{mt}}{\partial y_{\theta}}\right)  
    + \big(A_j\big)_{ki} + \big(A_i\big)_{kj} - \sum_u g^{ku}\frac{\partial g_{ij}}{\partial x_u} =0.
  \end{equation} 
    
To this aim, we write according to (\ref{expres}) 
$$ (A_j)_{ki}= \sum_{s,t,u,v,w} \frac{\partial F_k^{-1}}{\partial y_s} \circ F \, \tilde{g}^{st} \circ F \, \frac{\partial F_u^{-1}}{\partial y_t}\circ F\, \frac{\partial^2 F_v}{\partial x_j \partial x_u} \,\tilde{g}_{vw} \circ F \,\frac{\partial F_w}{\partial x_i}.$$

By writing $(A_i)_{kj}$ the same way (exchanging the role of $v$ and $w$) and then adding the two expressions, we get
\begin{equation}\label{Aterm}
(A_i)_{kj} + (A_j)_{ki} = \sum_{s,t,u,v,w} \frac{\partial F_k^{-1}}{\partial y_s} \circ F \, \tilde{g}^{st} \circ F \, \tilde{g}_{vw} \circ F \, \frac{\partial F_u^{-1}}{\partial y_t}\circ F \,\frac{\partial}{ \partial x_u} \Big(\frac{\partial F_w}{\partial x_i} \frac{\partial F_v}{\partial x_j}\Big).
\end{equation}

To conclude, it remains to rewrite $\sum_l g^{kl}\frac{\partial g_{ij}}{\partial x_l}$. Using (\ref{dermetriccoor}), we obtain 
$$ \frac{\partial g_{ij}}{\partial x_u} = \sum_{w,v} \frac{\partial}{\partial x_u} \big( \frac{\partial F_w}{\partial x_i} \frac{\partial F_v}{\partial x_j}\big) \tilde{g}_{wv} \circ F \, + \sum_{w,v} \frac{\partial F_w}{\partial x_i} \frac{\partial F_v}{\partial x_j} \frac{\partial }{\partial x_u} \big( \tilde{g}_{wv} \circ F\big).$$

Then, combining this together with (\ref{metriccoor}):
$$ g^{ku} = \sum_{s,t} \frac{\partial F_k^{-1}}{\partial y_s} \circ F \, \tilde{g}^{st} \circ F \, \frac{\partial F_u^{-1}}{\partial y_t} \circ F
$$
yield,
\begin{multline}\label{DF}
\sum_u g^{ku} \frac{\partial g_{ij}}{\partial x_u} = \underbrace{\sum_{u,w,v,s,t}  \frac{\partial F_k^{-1}}{\partial y_s} \circ F \, \frac{\partial F_u^{-1}}{\partial y_t} \circ F \, \frac{\partial}{\partial x_u} \big( \frac{\partial F_w}{\partial x_i} \frac{\partial F_v}{\partial x_j}\big) \tilde{g}_{wv} \circ F \, \tilde{g}^{st} \circ F }_{(D)}  \\
+  \underbrace{ \sum_{u,w,v,s,t}  \frac{\partial F_k^{-1}}{\partial y_s} \circ F \, \frac{\partial F_u^{-1}}{\partial y_t} \circ F \, \frac{\partial F_w}{\partial x_i} \frac{\partial F_v}{\partial x_j}
\frac{\partial }{\partial x_u} \big( \tilde{g}_{wv} \circ F\big) \tilde{g}^{st} \circ F
}_{(E)} .
\end{multline}

According to (\ref{Aterm}), $(A_i)_{kj} + (A_j)_{ki}= (D)$. To complete the proof, we apply (\ref{crf}) to $\tilde{g}_{wv}$  in $(E)$. This gives
\begin{align*}
(E) & =  \sum_{u,w,v,s,t,\theta}  \frac{\partial F_k^{-1}}{\partial y_s} \circ F \, \underbrace{\frac{\partial F_u^{-1}}{\partial y_t} \circ F \, \frac{\partial F_{\theta}}{\partial x_u}}_{\sum_u = \delta_{t\theta}}\frac{\partial F_w}{\partial x_i} \frac{\partial F_v}{\partial x_j}
 F^*\big( \frac{\partial \tilde{g}_{wv}}{\partial y_{\theta}} \big) \tilde{g}^{st} \circ F
 \\
 &=  \sum_{w,v,s,\theta} \frac{\partial F_k^{-1}}{\partial y_s} \circ F \,\frac{\partial F_w}{\partial x_i} \frac{\partial F_v}{\partial x_j}
 F^*\big( \frac{\partial \tilde{g}_{wv}}{\partial y_{\theta}} \big) \tilde{g}^{s\theta} \circ F.
\end{align*} 
 Therefore $(E)$ coincides with the first term in (\ref{topr}) and the proof is complete.

%%%%%%%%%%%%%%%%%%%%%%%%%%% Comments%%%%%%%%%%%%%%%%%%%%%%%%%%

%\newpage
%\input{comments} 

%%%%%%%%%%%%% bibliography %%%%%%%%%%%%%%%%%%%%%%%%%%%%%%%%%
%%%%%%%%%%%%%%%%%%%%%%%%%%%%%%%%%%%%%%%%%%%%%%%%%%%%%%%%
\bibliographystyle{plain}

\begin{thebibliography}{10}

\bibitem{Alex39}
A.~D. Alexandroff.
\newblock Almost everywhere existence of the second differential of a convex
  function and some properties of convex surfaces connected with it.
\newblock {\em Leningrad State Univ. Annals [Uchenye Zapiski] Math. Ser.},
  6:3--35, 1939.

\bibitem{AB}
Luigi Ambrosio and J{\'e}r{\^o}me Bertrand.
\newblock Work in progress.

\bibitem{ABSurface}
Luigi Ambrosio and J{\'e}r{\^o}me Bertrand.
\newblock On the regularity of {A}lexandrov surfaces with curvature bounded
  below.
\newblock {\em Preprint}, 2014.

\bibitem{AFP}
Luigi Ambrosio, Nicola Fusco, and Diego Pallara.
\newblock {\em Functions of bounded variation and free discontinuity problems}.
\newblock Oxford Mathematical Monographs. The Clarendon Press Oxford University
  Press, New York, 2000.

\bibitem{AGSAP}
Luigi Ambrosio, Nicola Gigli, and Giuseppe Savar{\'e}.
\newblock Bakry-\'{E}mery curvature-dimension condition and Riemannian
  Ricci curvature bounds.
\newblock  Annals of Probability, {\bf 43} (2015), 339--404.

\bibitem{Bakry-94}
D.~Bakry.
\newblock {\em On Sobolev and logarithmic Sobolev inequalities for
              Markov semigroups.}
\newblock New trends in stochastic analysis (Charingworth, 1994), World Sci. Publ., River Edge, NJ, 1997, 43--75.

\bibitem{bbi}
Dmitri Burago, Yuri Burago, and Sergei Ivanov.
\newblock {\em A course in metric geometry}, volume~33 of {\em Graduate Studies
  in Mathematics}.
\newblock American Mathematical Society, Providence, RI, 2001.

\bibitem{bgp}
Yuri Burago, Mikha\"il Gromov, and Grigori Perel{\cprime}man.
\newblock A. {D}. {A}leksandrov spaces with curvatures bounded below.
\newblock {\em Uspekhi Mat. Nauk}, 47(2(284)):3--51, 222, 1992.

\bibitem{evans}
Lawrence~C. Evans and Ronald~F. Gariepy.
\newblock {\em Measure theory and fine properties of functions}.
\newblock Studies in Advanced Mathematics. CRC Press, Boca Raton, FL, 1992.

\bibitem{federer69}
Herbert~Federer. 
\newblock {\em Geometric Measure Theory.} 
Springer, Berlin, 1969.

\bibitem{GHL}
Sylvestre Gallot, Dominique Hulin, and Jacques Lafontaine.
\newblock {\em Riemannian geometry}.
\newblock Universitext. Springer-Verlag, Berlin, third edition, 2004.

\bibitem{Gigli}
Nicola~Gigli.
\newblock {\em Nonsmooth Differential Geometry - An approach tailored for spaces with Ricci curvature
bounded from below.}
\newblock ArXiv Preprint 1407.0809, 2014.

\bibitem{hartman}
Philip Hartman.
\newblock On functions representable as a difference of convex functions.
\newblock {\em Pacific J. Math.}, 9:707--713, 1959.

\bibitem{kuwae}
Kazuhiro Kuwae, Yoshiroh Machigashira, and Takashi Shioya.
\newblock Sobolev spaces, {L}aplacian, and heat kernel on {A}lexandrov spaces.
\newblock {\em Math. Z.}, 238(2):269--316, 2001.

\bibitem{mantegazza}
Carlo Mantegazza, Mascellani Giovanni, and Uraltsev Gennady.
\newblock On the distributional {H}essian of the distance function.
\newblock {\em Arxiv, arXiv:1303.1421}.

\bibitem{OS}
Yukio Otsu and Takashi Shioya.
\newblock The {R}iemannian structure of {A}lexandrov spaces.
\newblock {\em J. Differential Geom.}, 39(3):629--658, 1994.

\bibitem{DC}
Grigori Perel{\cprime}man.
\newblock D{C} {S}tructure on {A}lexandrov {S}pace.
\newblock {\em Preprint}, 1994.

\bibitem{petersen}
Peter Petersen.
\newblock {\em Riemannian geometry}, volume 171 of {\em Graduate Texts in
  Mathematics}.
\newblock Springer, New York, second edition, 2006.

\bibitem{SCPet}
Anton Petrunin.
\newblock Semiconcave functions in {A}lexandrov's geometry.
\newblock In {\em Surveys in differential geometry. {V}ol. {XI}}, volume~11 of
  {\em Surv. Differ. Geom.}, pages 137--201. Int. Press, Somerville, MA, 2007.

\bibitem{LVPet}
Anton Petrunin.
\newblock Alexandrov meets {L}ott-{V}illani-{S}turm.
\newblock {\em M\"unster J. Math.}, 4:53--64, 2011.

\bibitem{SpivakI}
Michael Spivak.
\newblock {\em A comprehensive introduction to differential geometry. {V}ol.
  {I}}.
\newblock Publish or Perish, Inc., Wilmington, Del., second edition, 1979.

\bibitem{Sturm}
Karl-Theodor Sturm.
\newblock {\em Gradient Flows for Semiconvex Functions on Metric Measure Spaces - Existence, Uniqueness and Lipschitz Continuity.}
\newblock Arxiv Preprint 1410.3966.

\end{thebibliography}

\def\cprime{$'$} \def\cprime{$'$} \def\cprime{$'$} \def\cprime{$'$}
  \def\cprime{$'$} \def\cprime{$'$}

\end{document}